\documentclass[a4paper]{article}
\usepackage[T1]{fontenc}
\usepackage[utf8]{inputenc}
\usepackage{framed}
\usepackage[english]{babel}
\usepackage{amssymb}
\usepackage{amsthm}
\usepackage{amsmath}
\usepackage{stmaryrd,mathtools}
\mathtoolsset{showonlyrefs}
\usepackage{bbm}
\usepackage{tikz}
\usepackage{xfrac}
\usetikzlibrary{arrows,chains,matrix,positioning,scopes}
\usepackage{tikz-cd}
\usetikzlibrary{matrix,arrows,decorations.pathmorphing}
\usepackage{pdfpages}
\usepackage{graphicx}
\usepackage[footnote,draft,danish,silent,nomargin]{fixme}
\usepackage[bookmarks,colorlinks]{hyperref}
\usepackage{enumerate}

\usepackage{geometry}
\bibstyle{bibo}

\numberwithin{proposition}{section}

\numberwithin{setup}{section}

\numberwithin{statement}{section}

\numberwithin{conjecture}{section}
\numberwithin{definition}{section}
\newtheorem{theorem}{Theorem}
\numberwithin{theorem}{section}
\newtheorem{Lemma}{Lemma}
\numberwithin{Lemma}{section}
\newtheorem{corollary}{Corollary}
\numberwithin{corollary}{section}
\newtheorem{remark}{Remark}
\numberwithin{remark}{section}

\numberwithin{condition}{section}

\newcommand{\dv}{\text{d}}

\newcommand{\T}{\mathcal{T}}

\newcommand{\tr}{\mathrm{tr}}

\newcommand{\R}{\mathbb{R}}

\newcommand{\Z}{\mathbb{Z}}

\newcommand{\C}{\mathbb{C}}
\newcommand{\N}{\mathbb{N}}
\newcommand{\M}{\mathcal{M}}

\newcommand\emptyarg{{}\cdot{}}

\makeatletter
\newcommand*{\rom}[1]{\expandafter\@slowromancap\romannumeral #1@}
\makeatother

\begin{document}

\title{Asymptotic expansions of the \\ Witten-Reshetikhin-Turaev Invariants of Mapping Tori I} 

\author{Jørgen Ellegaard Andersen and William Elbæk Petersen \footnote{Work supported in part by the center of excellence grant Center for Quantum Geometry of Moduli Spaces from the Danish National Research Foundation (DNRF95)}}


\maketitle 

\markboth{ Jørgen Ellegaard Andersen and William Elbæk Petersen}{Asymptotic expansions of the WRT Invariants of Mapping Tori} 

%
%
%
%

\begin{abstract}
 In this paper we engage in a general study of the asymptotic expansion of the Witten-Reshetikhin-Turaev invariants  of mapping tori of surface mapping class group elements. We use the geometric construction of the Witten-Reshetikhin-Turaev TQFT via the geometric quantization of moduli spaces of flat connections on surfaces. We identify assumptions on the mapping class group elements that allow us to provide a full asymptotic expansion.  In particular, we show that our results apply to all pseudo-Anosov mapping classes on a punctured torus and show by example that our assumptions on the mapping class group elements are strictly weaker than hitherto successfully considered assumptions in this context. The proof of our main theorem relies on our new results regarding asymptotic expansions of oscillatory integrals, which allows us to go significantly beyond the standard transversely cut out assumption on the fixed point set. This makes use of Picard-Lefschetz theory for Laplace integrals.
\end{abstract}

\pagestyle{myheadings}

\section{Introduction}

The Witten-Reshetikhin-Turaev TQFT $Z_{\text{WRT}}^{(k)}$ \cite{ReshetikhinTuraev90,ReshetikhinTuraev91,Turaev}, originally constructed  combinatorially using representation theory of the quantum group $U_q(\mathfrak{sl}(2))$ at a $(k+2)$-th root of unity $q,$ was motivated by Atiyah's axioms \cite{Atiyah} and Witten's  work \cite{Witten} on the Jones polynomial \cite{Jones85,Jones} and quantum Chern-Simons theory. This TQFT was subsequently realized through Skein-theory by Blanchet, Habegger, Masbaum and Vogel \cite{Blanchet00,BHMV92,BHMV}. Witten's original work proposed a gauge theoretic construction through geometric quantization of moduli spaces, and this was done independently by Axelrod-della Pietra-Witten \cite{APW}, and Hitchin \cite{Hitchin} through methods later generalized to the context of general Hitchin connections by the first named author et al. \cite{Andersen12,AndersenGammelgaardRoed,AndersenGammelgaard14,AR}. Witten further conjectured in \cite{Witten} that the gauge theoretic approach should be equivalent to a construction using conformal field theory (CFT). The CFT approach was realized by Tscuchia et al \cite{TUY}, and through work of Laszlo \cite{La1} and work of the first named author and Ueno \cite{AU1,AU2,AU3,AU4}, it is now known that all three approaches; the gauge theoretic, the CFT approach, and the combinatorial, are equivalent. Through Witten's path integral motivation for the WRT-TQFTs, it is expected that they should admit asymptotic expansions in the level $k$, and this is known as the asymptotic expansion conjecture. For a summary of results, see \cite{Andersen12Crelle,AndersenHimpel12,CharlesMarhcI,CharlesMarchII,Joerg} and the references therein. 

In this paper we study the asymptotic expansion of the WRT invariants of mapping tori with (special) colored links using the gauge theoretic construction of the WRT-TQFT, which was also used by the first named author to prove asymptotic faithfulness \cite{Andersen06}. We start by describing the so-called co-prime case. Let $\Sigma$ be the complement of a point $P$ on an oriented surface of genus at least two $\bar\Sigma$. Let $\Gamma_{\bar\Sigma,v_P}$ be the mapping class group of $(\Sigma,P,v_P)$, where $v_P$ is a tangent direction at $P$ and let $\varphi\in \Gamma_{\bar\Sigma,v_P}$. Let $\mathcal{M}$ be the moduli space of flat $\text{SU}(n)$ connections on $\Sigma,$ whose holonomy around the puncture is a fixed generator $\delta \in Z(\text{SU}(n))$ of the center of $\text{SU}(n).$ This is a compact, simply connected, symplectic manifold equipped with the Narasimhan-Atiyah-Bott-Goldman form $\omega$ \cite{AtiyahBott,Goldman84}, and $\M$ supports the Chern-Simons line bundle $\mathcal{L}_{\text{CS}},$ which is a hermitian line bundle with a unitary connection with curvature $-i\omega$ \cite{Freed}.  A choice of complex structure $\sigma$ on $\Sigma$ induces a Kähler manifold structure on $\mathcal{M}$ through the work of Narasimhan and Seshadri \cite{NS}. The Verlinde bundle is the vector bundle over Teichmüller space $\mathcal{V}^{(k)}\rightarrow \mathcal{T}_{\Sigma},$ whose fiber at a point $\sigma$ is $H^0\left(\M_{\sigma},\mathcal{L}_{\text{CS}}^{\otimes k}\right),$ and the module of states $Z^{(k)}_{\text{CS}}(\Sigma)$ associated to $\Sigma$ is identified with any fibre of this bundle. The Hitchin connection constructed in  \cite{APW,Hitchin} provides the identification (projectively) of the different fibers. The mapping class $\varphi$  act on $\mathcal{M}$, and this lifts to an action $\varphi^{(k)}$ on the Verlinde bundle. Composing $\varphi^{(k)}$ with the parallel transport of the Hitchin connection we get the so-called quantum action on $Z^{(k)}_{\text{CS}}(\Sigma),$ denoted by $Z_{\text{CS}}^{(k)}(\varphi).$  Consider the mapping torus $T_{\varphi}=\Sigma\times I/[(x,0)\sim (\varphi(x),1)].$ From the axioms of a TQFT we have
\begin{equation} \label{MFandMappingTorus}
Z_{\text{WRT}}^{(k)}(T_{\varphi})=\tr\left(Z_{\text{WRT}}^{(k)}({\varphi}): Z_{\text{WRT}}^{(k)}(\Sigma) \rightarrow Z_{\text{WRT}}^{(k)}(\Sigma)\right).
\end{equation}
Via the isomorphisms described above, this formula motivates the study of the asymptotic expansion of $\tr\left(Z_{\text{CS}}^{(k)}(\varphi)\right).$ Let $\M^{\varphi}=\{z \in \M: \varphi(z)=z\}$ and let $\M_{T_{\varphi}}$ be the moduli space of flat connections in a principal $\text{SU}(n)$-bundle on $ T_{\varphi} \setminus L$ with central meridional holonomy $\delta$ along the link $L$ in the mapping torus traced out by the puncture of $\Sigma$.  We have a natural surjection $\rho: \M_{T_{\varphi}} \rightarrow \M^{\varphi}.$ Let $S_{\text{CS}}$ be the the Chern-Simons action \cite{AJHM,Freed}. Let $2n_0=\text{dim}(\M).$   Setting $r=k+n$, we prove the following theorem with no assumptions on the mapping class.
 
 \begin{theorem} \label{1.1} For any mapping class $\varphi$, there exists a smooth $\C/ 2\pi i \Z$-valued function $P_{\varphi}$ on ${\mathcal M}$ together with a sequence of smooth top forms $\{\Omega_n\}$ such that for every $N\in \N$ we have \begin{equation}\begin{split}
\tr \left(Z^{(k)}_{\text{CS}}(\varphi)\right) =r^{n_0} \sum_{n=0}^{N}r^{-n}\int_{\mathcal M} e^{rP_{\varphi}} \ \Omega_n   +O\left(k^{n_0-(N+1)}\right).\end{split}
\end{equation}
$P_{\varphi}$ is real analytic near $\M^{\varphi},$  and the real part of $P_\varphi$ is strictly negative away from $\M^{\varphi}.$ Every $z \in \M^{\varphi}$ is a stationary point of $P_{\varphi}$ and $$P_{\varphi} \circ \rho=2\pi iS_{\text{CS}}.$$
\end{theorem}
We emphasize that the phase $P_\varphi$ is a globally defined function on $\mathcal M$. Under assumptions on the action of $\varphi$ on the moduli space, we can use this theorem to obtain more information on the asymptotic expansion of the quantum invariant. Let $\{\theta_j\}$ be the critical Chern-Simons values, and for each $j,$ let $2m_j$ be the maximal value of $\text{dim}(\text{Ker}(\dv \varphi_z-I))$ among $z \in \M^{\varphi}$ with $P_{\varphi}(z)=2\pi i\theta_j.$ Recall that $\M^{\varphi}$ is said to be cut out transversely if any connected component $Y\subset \M^{\varphi}$ is smooth and satisfies $T Y=\text{Ker}(\dv \varphi-I)_{\mid Y}.$ Observe that the real analyticity of $P_{\varphi}$ guarantee that for any $z \in \M^{\varphi}$ there exists a coordinate neighborhood $U$ of $z,$ such that $P_{\varphi}$ is the restriction of a holomorphic function $\check{P}_{\varphi} $ defined on $U+\sqrt{-1} U.$ Thus we can invoke the theory of Laplace integrals with holomorphic phase and Picard-Lefschetz theory, pioneered by Malgrange \cite{Malgrange} and Pham  \cite{Pham83}. For this; no assumption on $\text{Hess}(P_{\varphi})$ is needed. We prove the following theorem.

\begin{theorem} \label{MainTheorem2} If $\M^{\varphi}$ is cut out transversely there exists a sequence of differential forms $\Omega_{\alpha}^j$ on $\M^{\varphi}$ giving an asymptotic expansion \begin{equation}
 \tr \left(Z_{\text{CS}}^{(k)}(\varphi)\right)
\sim  \sum_{j} e^{ 2\pi i r \theta_j} r^{m_j} \sum_{\alpha=0}^{\infty} r^{-\frac{\alpha}{2}} \int_{\mathcal{M}^{\varphi}} \Omega^j_{\alpha}.
\end{equation} 
If all $z \in \M^{\varphi}$ satisfy one of the following three conditions \begin{itemize}  
\item \label{2} $z$ is a smooth point with $T_z \M^{\varphi}=\text{Ker}(\dv\varphi_z-I),$
\item  \label{3} $\text{dim}(\text{Ker}(\dv \varphi_z -I)) \leq 1,$ or
\item  $z$ is an isolated stationary point of the germ of the holomorphic extension $\check{P}_{\varphi}$ at $z,$ 
\end{itemize}  we then have an asymptotic expansion of the form \[
\tr \left(Z_{\text{CS}}^{(k)}(\varphi)\right) \sim  \sum_j e^{2\pi i r\theta_j}r^{n_j}\sum_{\alpha \in A_j, \beta \in B_j} c_{\alpha,\beta} r^{\alpha} \log(r)^{\beta}. 
 \]
  Here $n_j \in \mathbb{Q}_{\geq 0},$  each $A_j\subset \mathbb{Q}_{\leq 0}$ is a union of finitely many arithmetic progressions, and each $B_j \subset \mathbb{N}$ is finite. If for all $z \in \mathcal{M}^{\varphi}$ the first or second condition holds, then $B_j=\{0\}$ and $n_j=m_j$ for all $j.$
\end{theorem}
All asymptotic expansions in this paper are in the Poincaré sense \cite{Olver} and are denoted by $\sim$ , i.e. if the conditions of Theorem \ref{MainTheorem2} holds, we have for any  $N\in \mathbb{N}$
$$\tr \left(Z_{\text{CS}}^{(k)}(\varphi)\right)= \sum_j e^{2\pi i r\theta_j}r^{n_j}\sum_{\alpha \in A_j(N), \beta \in B_j} c_{\alpha,\beta} r^{\alpha} \log(r)^{\beta}+O\left(k^{-N}\right), $$ where $A_j(N)=A_j\cap [-N-n_j,0].$ The numbers $n_j$ and the coefficients $c_{\alpha,\beta}$ as well as the exponents $\alpha,\beta$ are uniquely determined. If $\M^{\varphi}$ is transversely cut out, then $\M_{T_{\varphi}}$ is smooth and $2m_j$ is the maximal dimension of a component $Y\subset \M_{T_\varphi}$ on which ${S_{\text{CS}}}_{\mid Y}=\theta_j.$  

As an example we consider the moduli space $\mathcal{M}_l$ of flat $\text{SU}(2)$ connections on a punctured torus $\Sigma_1^1,$ whose holonomy around the puncture has trace $l\in (-2,2).$  For any $\varphi$ in the mapping class group $\Gamma_{(\bar\Sigma_1^1, v_P)}$ one can construct $Z_l^{(k)}({\varphi})$ in a manner analogue to the quantum actions described above. The analog of Theorem \ref{1.1} also holds in the punctured torus case, and we denote the associated phase functions by $P_{\varphi,l}.$ The relevant projectively flat Hitchin connection, which matches with the TUY-connection under the Pauly isomorphism \cite{Pauly1}, has not been identified among all the possible Hitchin connections, which exists in this parabolic case. This is however not a concern for us, since we prove for each of these Hitchin connections the following theorem, where $\{\theta_j\}$ denotes the Chern-Simons values discussed in \cite{AJHM}, $m_j$ is as above.
\begin{theorem} \label{NiceResult}
If all $z \in \M_l^{\varphi}$ satisfy one of the following three conditions \begin{itemize}  
\item \label{4} $z$ is a smooth point with $T_z \M_l^{\varphi}=\text{Ker}(\dv\varphi_z-I),$
\item  \label{5} $\text{dim}(\text{Ker}(\dv \varphi_z -I)) \leq 1,$ or
\item  $z$ is an isolated stationary point of the germ of the holomorphic extension $\check{P}_{\varphi,l}$ at $z,$
\end{itemize}  we then have an asymptotic expansion of the form \[
\tr \left(Z_{l}^{(k)}(\varphi)\right) \sim  \sum_j e^{2\pi i r\theta_j}r^{n_j}\sum_{\alpha \in A_j, \beta \in B_j} c_{\alpha,\beta} r^{\alpha} \log(r)^{\beta}. 
 \]
 Here $n_j \in \mathbb{Q}_{\geq 0},$ each $A_j\subset \mathbb{Q}_{\leq 0}$ is a union of finitely many arithmetic progressions, and each $B_j \subset \mathbb{N}$ is finite. If for all $z \in \mathcal{M}_l^{\varphi}$ the first or second condition holds, then $B_j=\{0\}$ and $n_j=m_j$ for all $j.$ If $\varphi \in \Gamma_{(\bar\Sigma_1^1, v_P)}$ is a pseudo-Anosov homeomorphism then for almost all $l\in (-2,2)$ the first condition holds for all $z \in \M_{l}^{\varphi}$ and  we have an expansion of the form \begin{equation}
  \tr \left(Z_{l}^{(k)}(\varphi)\right)
\sim  \sum_{j} e^{ 2\pi i r \theta_j} \sum_{\alpha=0}^{\infty} c_{j,\alpha} r^{-\frac{\alpha}{2}}.
\end{equation} 
\end{theorem} 
We provide an explicit example of an element $\varphi \in \Gamma_1^1$ such that $\M_{-1/4}^{\varphi}$ is not transversely cut out, but which satisfy the condition of Theorem \ref{NiceResult}. The proofs of Theorem \ref{MainTheorem2},  and Theorem \ref{NiceResult} proceeds by an application of stationary phase approximation to the integrals of Theorem \ref{1.1}. In doing so, in the case where $\M^{\varphi}$ is not cut out transversely, we prove a result on stationary phase approximation by combining work of Hörmander and Malgrange. See Remark \ref{Comparison} for a comparison.

\begin{theorem}\label{AnalyticalResult2}
Let $f \in C^{\infty}(\mathbb{R}^n,\mathbb{C}).$ Assume $f$ has a stationary point $p,$ which is a maximum of the real part of $f$. Assume $\text{Re}(f)$ and $\text{Im}(f)$ are both real analytic near $p.$ Let $\phi$ be a smooth function with compact supported contained in a small neighborhood $D$ of $p.$ If $\text{Hess}(f)_p$ is non-degenerate on a subspace of $T_p \R^n$  of co-dimension one and $D$ is sufficiently small, then there exists\footnote{See Remark \ref{MMM} for more details on $m.$} $m\in \mathbb{N}$ such that we have an asymptotic expansion 
\begin{equation} \label{ourresult}
\int_{\R^n} e^{kf(x)} \phi(x) \ \dv x \sim e^{kf(p)} k^{-\frac{n-1}{2}} \sum_{\alpha=0}^{\infty} c_\alpha(\phi) k^{-\alpha/m}.
\end{equation}
Write $a=\text{Re}(f)$. If the holomorphic extension $\check{f}$ has an isolated stationary point at $p$ and $a-a(p)$ has an isolated zero at $p,$ and if $D$ is sufficiently small, then there exists a union $A$ of finitely many arithmetic progressions of negative rational numbers and a finite set $B\subset \mathbb{N}$ such that we have \begin{equation} \label{ourresult2}
 \int_{\R^n} e^{kf(x)} \phi(x) \ \dv x \sim  e^{kf(p)} \sum_{\alpha \in A, \beta \in B} c_{\alpha,\beta}(\phi) k^\alpha \log(k)^\beta.
 \end{equation}\end{theorem}


 We have stated our results in terms of topological quantum field theory, but our asymptotic analysis is carried out within the general formalism of geometric Kähler quantization and Hitchin connections, which provides a framework for relating different quantizations of a symplectic manifold, stemming from different choices of Kähler structures. This is a topic of intrinsic interest. We work with a quadruple $(\Gamma,\T,L,M)$ consisting of a prequantum line bundle $L\rightarrow M$ equipped with a family of Kähler structures on $M$ parametrized by a manifold $\T,$ and an action of a group $\Gamma$ by symplectomorphisms on $M$ together with an equivariant action on $\T.$ We call this data a pre-quantum action. Choosing a Kähler structure on $M$ within the family $\T,$ we get a quantization by taking global holomorphic sections of $L^{\otimes k}.$ We impose conditions made by the first named author et al. in \cite{Andersen12,AndersenGammelgaardRoed,AndersenGammelgaard14,AR}, ensuring that the resulting different quantizations form a smooth vector bundle over $\T,$ which supports a projectively flat connection known as a Hitchin connection. The action of $\Gamma$ lifts to $L,$ and by composing with the monodromy of the Hitchin connection, we can define a projective representation $Z^{(k)}$ of $\Gamma,$ which we call a quantum representation. The results of this paper are achieved by analyzing the asymptotic expansion of the character of a quantum representation, and then applying our results to the quantum representations of the mapping class groups  which arise in topological quantum field theory as described above. The procedure described above is related to the problem of quantizing a symplectomorphism. This has been considered by several authors in other contexts, see for instance the works of Zeldith \cite{Zeldith97,Zeldith03}, Bolte \cite{Bolte98}, or the works of Galasso and Paoletti \cite{GalassoPaolette18}.

 The appearance of oscillatory integrals where the phase is the restriction of a holomorphic function in Theorem \ref{1.1} is in agreement with recent developments in TQFT and quantum field theory \cite{DunneUnsal15,Kontsevich12}, which makes use of Écalle´s theory of resurgence \cite{Ecalle81a,Ecalle81b}. The investigation of the role played by the resurgence in relation to the WRT-TQFT has been by pioneered in works of Witten \cite{Witten2010} and Garoufalidis \cite{Garoufalidis08}, and more recently in works by Gukov-Putrov-Marinö \cite{GukovMarinoPutrov} and Chun \cite{Chun17} which sheds light on the relation between WRT-invariants and number theory discovered by Lawrence-Zagier \cite{LawrenceZagier} and further explored by Hikami \cite{Hikami2005}.

\begin{remark} Our result in the case where $\M^{\varphi}$ is transversally cut out can be seen as a generalization of a result due to Charles \cite{Laurent10,Charles16}. In \cite{Charles16} he carries out an asymptotic analysis of $\tr(\tilde{\rho}^{(k)}(\varphi)),$ where $\tilde{\rho}^{(k)}$ is a lift of the (projective) quantum representation defined using the monodromy of the Hitchin connection, to what he calls an asymptotic representation of a central extension $\tilde{\Gamma}$ of the mapping class group of a punctured surface of genus at least two. The main result in \cite{Charles16} provides the leading order asymptotic under the condition that $\M_{\Sigma}^{\varphi}$ is a $0$-dimensional submanifold and $\dv \varphi_x-I$ is invertible for every fixed point $x.$ We emphasize, that in order to prove the existence of asymptotic expansions without the transversally cut assumption, we appeal to Theorem \ref{AnalyticalResult2}. See also the very recent work of Louis Ioos \cite{Ioos}, which gives an explicit computation of the leading order coefficient in the expansion, in the case where $\M^{\varphi} $ is transversely cut out.
\end{remark}

Let us finally comment on the projective ambiguity of $Z^{(k)}_{\text{CS}}.$ This discussion draws upon recent work by the first named author and his former PhD student Skovgaard Poulsen. In their joint work \cite{ASP16,ASP17,ASP18} it is proven that the level $k$ Hitchin connection has curvature 
\[
F_{\nabla^{H,k}}=\frac{ik(n^2-1)}{12(k+n)\pi}\omega_{\T},
\] where $\omega_\T$ is the Weil-Petersson K\"{a}hler form. The construction of $Z^{(k)}_{\text{CS}}(\varphi)$ depends on the choice of a complex structure on $\Sigma,$ as well as a choice of a curve $\gamma$ from $\varphi^*\sigma$ to $\sigma$ in Teichmüller space, along which, we compute the parallel transport with respect to the Hitchin connection. Using the result of \cite{ASP16,ASP17,ASP18}, it follows that the sequence $Z_{\text{CS}}^{(k)}(\varphi)$ admits an asymptotic expansion as $k$ tends to $+\infty$ independent of the choice of the above mentioned choice of the curve $\gamma$.

 This paper is organized as follows. \textbf{Section \ref{Holomorphic quantization}} introduces the general setup of geometric quantization and Hitchin connections, as well as pre-quantum actions, in the general context of \cite{Andersen06,Andersen12}. \textbf{Section \ref{Quantum Characters}} relates the large $k$ asymptotic of the character for a pre-quantum action to the large $k$ asymptotic of certain oscillatory integrals, and identifies conditions that ensure the existence of asymptotic expansions of these followed by a proof of Theorem \ref{AnalyticalResult2}. \textbf{Section \ref{TQFT}} relates the general results obtained in the previous section to the case of moduli spaces and topological quantum field theory, thus proving Theorem \ref{1.1}, Theorem \ref{MainTheorem2} and Theorem \ref{NiceResult}. In an appendix we provide the computational details concerning the example of mapping class for which the fixed point set is not cut out transversely, but for which our results applies.

\paragraph*{Acknowledgements.} The authors wish to thank Søren Fuglede Jørgensen for initiating the second named author to parts of the relevant setup for this paper.

\section{The Hitchin connection and quantum representations} \label{Holomorphic quantization}

We recall the basic setup of \cite{Andersen06,Andersen12}. A \textit{pre-quantum line bundle} is a line bundle $(L,h,\nabla)\rightarrow (M,\omega),$ where $(M,\omega)$ is a symplectic manifold of real dimension $2n_0$, $h$ is a hermitian metric, and $\nabla$ is a unitary connection on $L$ with curvature $-i\omega.$ From now on, we shall assume that $M$ is compact, $b_1(M)$ vanishes and that there exists an integer $n_1,$ such that the first Chern class of $(M,\omega)$ equals $n_1[\omega]$. Assume that  $J$ is a complex structure on $M$ such that $(M,\omega,J)$ is a Kähler manifold denoted $M_J$. The \textit{$level$ $k$ quantum Hilbert space} associated to $(L,h,\nabla)\rightarrow (M,J,\omega)$ by \textit{geometric Kähler quantization} is given by the space of global holomorphic sections 
\begin{equation}
H^{(k)}:=H^0\left(M_J, L^{\otimes k}\right).
\end{equation}
This is equipped with the $L_2$ metric with respect to the volume form $\Omega=\frac{\omega^n}{n!(2\pi)^n}.$ A \textit{family of Kähler structures} is a smooth map $J:\T\rightarrow C^{\infty}(M,\text{End}(T M)),$ where $\T$ is a smooth manifold, such that $M_\sigma := (M,\omega,J(\sigma))$ is a Kähler manifold for each $\sigma \in \T.$ Given such a family, we consider the trivial bundle $ \mathcal{H}^k=\T\times C^{\infty}(M,L^k)\rightarrow \T,$ and we assume there is a subbundle $H^{(k)} \rightarrow \T,$ whose fiber at a point $\sigma$ is given by
\begin{equation}
H_\sigma^{(k)}=H^0( M_{\sigma}, L_{\sigma}^k).
\end{equation}
  We consider connections in $\mathcal{H}^k$ of the form
$
\nabla^t+u,
$
where $\nabla^t$ is the trivial connection, and $u$ is a $1$-form on $\mathcal{T}$ with values in differential operators acting on smooth sections of $L^k.$
A connection in $\mathcal{H}^k$ of this form is called a \textit{Hitchin connection} if it preserves $H^{(k)}$ \cite{Andersen12,AndersenGammelgaard14,AR}. We shall assume that the family 
$$J:\T \rightarrow C^{\infty}(M,\text{End}(TM)),$$ satisfies a condition called \textit{rigidity}, which is defined in \cite{Andersen12}. By the results of \cite{Andersen12} this will ensure that for each $k$ there exists a Hitchin connection, that we shall denote by $\nabla^{H,k}.$ Furthermore, we shall assume that for each $\sigma \in \T,$ the Kähler manifold $M_\sigma$ admits no non-zero global holomorphic vector fields. By the results of joint work by the first named author and Gammelgaard \cite{AndersenGammelgaard14}, this will ensure that the Hitchin connection is projectively flat.

 A \textit{pre-quantum action} of a group $\Gamma$ on $(\T,M,L)$ consists of an action of $\Gamma$ on $M$ by symplectomorphisms and an action of $\Gamma$ on $\T$ by diffeomorphisms, which are compatible in the sense that for all $\varphi \in \Gamma$ and for all $\sigma \in \T$ the associated symplectomorphism $\varphi:(M,\sigma)\rightarrow (M,\varphi.\sigma)$ is a biholomorphism. We note that for each $\varphi \in \Gamma$ and each $k \in \N$, there is a unique (up to a $U(1)$ multiplicative factor) unit norm smooth section
\begin{equation}
\tilde{\varphi} \in C^\infty(M, \varphi^*\left( L \right) \otimes  L^*),
\end{equation}
which is parallel with respect to the connection naturally induced by $\nabla$. Its $k$'th tensor power  induce a lift of the action of $\varphi$ on $M$ to a smooth bundle morphism
$$ \varphi^{(k)} : L \rightarrow L,$$
which therefore also induces a lift of $\varphi's$ action on $\mathcal T$ to a smooth bundle morphism
\begin{equation}
\varphi^{(k)}: H^{(k)} \rightarrow H^{(k)},
\end{equation}
that preserves the Hitchin connection on $H^{(k)}.$

If we have \textit{a real analytic structure on $(M,L)$} compatible with the family of complex structures $\mathcal T$ we can of course require that the pre-quantum action preserves this real analytic structure. We will see that this will allow us to obtain more control over the asymptotic expansions.

For any $k\in \N$ and any $\sigma \in \T$ we have a homomorphism
\begin{equation} \label{Projrep}
 \Gamma \rightarrow \ \text{PGL}\left(H_{\sigma}^{(k)}\right),
\end{equation} which for each element $\varphi \in \Gamma$ admits a representative
\begin{equation}
Z^{(k)}(\varphi) \in \text{GL}\left(H_{\sigma}^{(k)}\right),
\end{equation} constructed as follows. Choose a smooth curve $\gamma:I \rightarrow \T$ starting at $\varphi.\sigma$ and ending at $\sigma.$ Let $P_{\gamma}^{(k)}$ be parallel transport of the level $k$ Hitchin connection along $\gamma.$ By definition $Z^{(k)}(\varphi)$ acts via the composition
\begin{equation} \label{constructionofaction}
H^{(k)}_\sigma  \overset{{\varphi}^{(k)}}{\longrightarrow} H^{(k)} _{\varphi. \sigma} \overset{P_{\gamma}^{(k)}}{\longrightarrow} H^{(k)}_\sigma.
\end{equation}
The homomorphism \eqref{Projrep} is referred to as the \textit{quantum representation}.

Let $L^2(L^k)$ be the complex Hilbert space of square integrable sections of $L^k$, and let the \textit{$k$-th Bergman projector}  be the orthogonal projection
\begin{equation}
\pi^{(k)}_\sigma: L^2(L^k) \rightarrow H_\sigma^{(k)},
\end{equation}
whose kernel $K_\sigma^{(k)} \in C^{\infty}(M\times M , L^k \boxtimes (L^*)^k)$ is the so called \textit{$k$-th Bergman kernel}, the asymptotic of which was presented in joint work by Karabegov and Schlichenmaier \cite{KS01}. Their results are formulated  in terms of the kernel $B_{\sigma}^{(k)}\in C^{\infty}(X\times X)$ of the projector of $L^2(X)$ onto the Hardy space of $k$-homogeneous holomorphic functions on $X,$ the $U(1)$-bundle of $L^*.$ To adapt their results to our setting, we note that for two unit norm sections $\psi:U_1,\rightarrow L,\phi:U_2 \rightarrow L$  we have for all $p_1 \in U_1$ and $p_2 \in U_2$ that
\begin{equation} \label{relatingkernels}
K_{\sigma}^{(k)}(p_1,p_2)=B_{\sigma}^{(k)}(\psi^*(p_1),\phi^*(p_2))\left(\psi^k(p_1)\right)\otimes \left(\phi^k(p_2)\right)^*,\end{equation} where for a frame $\eta$ of $L,$ we let $\eta^*$ be the dual co-frame. Let $s: U \rightarrow L$ be a non-vanishing local holomorphic section on a contractible complex coordinate neighborhood $U$ with holomorphic coordinates $z.$ Let $\alpha= s/\lvert s\rvert.$ 
 Define $\Phi$ by
\begin{equation}\label{Phi}
\Phi=\log(h'(s^*)),
\end{equation}
where $h'$ is the induced metric on $L^*.$ 
Write $ y=(y_1,y_2)$ for the holomorphic coordinates on $U\times U$ naturally induced by the coordinates $z,$ and let $\Delta:U \rightarrow U\times U$ be the diagonal map. Let
 \begin{equation} \label{tildePhi} \tilde{\Phi} \in C^{\infty}(U\times \overline{U}),\end{equation} 
 be an almost analytical extension of $\Phi$, e.g. \begin{equation} \label{diagonal}
\tilde{\Phi}\circ \Delta=\Phi,
\end{equation}
and for $(v_{1},...,v_{m})$ any (possible empty) string with $v_i \in \{y_1,\overline{y}_1,y_2,\overline{y}_2\}$, we have 
\begin{equation}\begin{split}\label{vanishatDIA}
0=\frac{\partial^{m+1} \tilde{\Phi}}{\partial v_1 \cdots \partial v_m \partial \overline{y}_1}\circ \Delta
=\frac{\partial^{m+1} \tilde{\Phi}}{\partial v_1 \cdots \partial v_m \partial y_2}\circ \Delta.
\end{split}
\end{equation}
Let $\pi_i:U\times U \rightarrow U$ be the projection on each of the two factors $i=1,2$. We define $\chi$ as follows
\begin{equation}\label{chi}
\chi=\tilde{\Phi}-1/2(\Phi\circ \pi_1+\Phi\circ \pi_2).
\end{equation}
Shrinking $U$ further if necessary, we can assume that for $p_1\not=p_2$ we have
\begin{equation}\label{nonpositiveofchiawayfromdiag}
\text{Re}(\chi)(p_1,p_2)<0.
\end{equation} As stated in Theorem $5.6$ in \cite{KS01}, there exists a sequence $\{\tilde{b}_v\}$ of smooth functions such that for any compact subset $E \subset U \times U$ and $N\in \N$ one has 
\begin{equation} \label{expandingBergmankernel}
\sup_{(p_1,p_2) \in E} \left\lvert B_{\sigma}^{(k)}(\alpha^*(p_1),\alpha^*(p_2))-k^{n_0} e^{k\chi(p_1,p_2)} \sum_{v=0}^{N-1} k^{-v}\tilde{b}_v(p_1,p_2) \right\rvert =O\left(k^{n_0-N}\right).
\end{equation}
The functions $\tilde{b}_v$ are related to work of Zelditch \cite{Zelditch}. We now define the Toeplitz operator associated to any smooth function $f \in C^{\infty}(M)$ as the composition of the multiplication operator $M_f: L^2(L^k)\rightarrow L^2(L^k),$ given by $M_f(s) = fs,$ with the Bergman projection $\pi^{(k)}_\sigma$
\begin{equation}
T_{f,\sigma}^{(k)}=\pi^{(k)}_\sigma \circ M_f: L^2(L^k) \rightarrow H^{(k)}_{\sigma}.
\end{equation}
 Let $\sigma_1,\sigma_0 \in \T$ and $f \in C^{\infty}(M,\C).$ We introduce the following notation
\begin{equation}
T_{f,(\sigma_0,\sigma_1)}^{(k)}={T_{f,\sigma_1}^{(k)}}_{\mid H_{\sigma_0}^{(k)}}:H_{\sigma_0}^{(k)} \rightarrow H_{\sigma_1}^{(k)}.
\end{equation}
For references on Toeplitz operator theory and how it is related to deformation quantization please see \cite{KS01,AMS} and the following works of Bordemann-Meinrenken, and Schlichenmaier \cite{BordemannMeinrenken,Schlichenmaier2000,Schlichenmaier2001}. For further important work on semi-classical aspects of Toeplitz operator theory related to this paper we refer to the works of Boutet de Monvel and Guillemin  and Boutet de Monvel and Sjöstrand \cite{BdMG,BdMS}.

\section{Asymptotic expansions of quantum characters} \label{Quantum Characters}

Let $\varphi \in \Gamma,$ choose $\sigma_1\in \T$ and let $\gamma:I \rightarrow \T$ be a smooth curve starting at $\sigma_0=\varphi.\sigma_1$ and ending at $\sigma_1.$  We are interested in calculating the asymptotic of $\tr (Z^{(k)}(\varphi))$ as $k \rightarrow +\infty.$

\subsection{Quantum characters as oscillatory integrals}

Our first step is to analyze the asymptotic behavior of $P_{\gamma}^{(k)}$ by expanding this operator in  $\tilde{k}$ with coefficients in Toeplitz operators, thus linking the Hitchin connection with the Bergman projector. Here
$
\tilde{k}=k+\frac{n_1}{2},
$
where $n_1$ is as above. The reason we work with this shift of $k$ stems from the fact that this shift occurs in the Hitchin connection \cite{Andersen06,Andersen12,Hitchin}. In \cite{Andersen06} the first named author proves that there exists a sequence of smooth functions $\{f_u\}$ such that with respect to the operator norm, we have for all integers $m$ that
\begin{equation}\label{VIP}
\lVert P_{\gamma}^{(k)}-\sum_{u=0}^m T^{(k)}_{f_u,(\sigma_0,\sigma_1)}{\tilde k}^{-u} \rVert =O\left(k^{-(m+1)}\right).
\end{equation}
Theorem 6 in \cite{Andersen06} states this theorem for $m=1$, and the proof of Theorem 6 in that paper is easily seen to also prove this stronger version holding for all integers $m$.

Let us now cover $M^{\varphi}=\{x \in M: \varphi(x)=x\}$ by finitely many holomorphic coordinate charts 
\begin{equation} \label{U_w}
\{U_w\},
\end{equation} such that $L_{\mid U_w}$  admits a holomorphic frame $s_w$ and define for each $w$ the functions $\Phi_w,\tilde{\Phi}_w,$ and $\chi_w$ as above in \eqref{Phi}, \eqref{tildePhi} and \eqref{chi}. Let $\alpha_w$ be the norm $1$ normalization of $s_w.$ As $\tilde{\varphi}$ is unitary, there is a smooth $\R$-valued function $\theta_w$ such that $\tilde{\varphi}(\alpha_w)=\exp(i\theta_w)\varphi^*(\alpha_w).$ Define $R=(I,\varphi):M \rightarrow M \times M$ and the important smooth function $P_w^{\varphi}=P_w$ defined on $U_w$ by
\begin{equation} \label{Phase}
P_w=i\theta_w+\chi_w\circ R.
\end{equation}
We are now ready to expand $\tr\left(Z^{(k)}(\varphi)\right)$ as a sum of products of powers of $\tilde{k}$ and oscillatory integrals.

\begin{theorem}\label{QIasOI}
There exists a sequence of smooth compactly supported top forms $\Omega_n^w \in \Omega^{2n_0}(U_w) $ giving an asymptotic expansion \begin{equation}\label{lolExpansion of Trace}\begin{split}
\tr \left(Z^{(k)}(\varphi)\right) =\tilde{k}^{n_0} \sum_{n=0}^{N} \sum_{w }\left( \int_{U_w}  e^{\tilde{k}P^{\varphi}_w} \ \Omega^w_{n} \right) \tilde{k}^{-n} +O\left(k^{n_0-(N+1)}\right), \end{split}
\end{equation}
for each $N \in \N.$ If the pre-quantum action is real analytic, there is a function $P^{\varphi}: V \rightarrow \C / 2\pi i \Z,$ where $V$ is an open neighbourhood of $M^{\varphi},$ whose imaginary part and real part are both real analytic, and which satisfies
\begin{equation}
P^{\varphi}_{\mid {U_w}}= P^{\varphi}_w \ \text{mod} \ 2\pi i \Z.
\end{equation}
\end{theorem}

\begin{proof}
By \eqref{VIP} we have
 \begin{equation}\label{FirstEstimate}
 \tr \left(Z^{(k)}(\varphi)\right)= \sum_{u=0}^m \tr\left( T^{(k)}_{f_u,(\sigma_0,\sigma_1)}\circ \varphi^{(k)}\right)\tilde{k}^{-u}+O\left(k^{-(m+1)}\right).
 \end{equation}

As $\varphi$ is a symplectomorphism, the following formula is valid for any holomorphic section $s$ of $L^k$ and any $x\in M.$
\begin{align*}
\begin{split}
&(\pi_{\sigma_1}^{(k)} \circ (f \varphi^{(k)}) (s))(x)
\\&=\int_{M} K_{\sigma_1}^{(k)}(x,y) f(y) \tilde{\varphi}^{\otimes k}(s)(\varphi^{-1}(y)) \ \Omega(y)
\\&=\int_{M} K_{\sigma_1}^{(k)}(x,\varphi(y)) f(\varphi(y)) \tilde{\varphi}^{\otimes k}(s)(y) \ \Omega(y).
\end{split}
\end{align*}
An operator $P=\int B(x,y) \ \dv y$ given by a smooth integral kernel is trace class and its trace is given by
$$\tr (P)= \int B(y,y) \ \dv y.$$ From equation $5.3$ in \cite{KS01} it follows that away from the diagonal, the Bergman kernel is locally uniformly $O\left(k^{-N} \right)$ for every $N \in \N.$ By choosing a partition of unity $ (\mu_w)$ subordinate to the cover $\{ U_w\}$ of $M^\varphi$ and combining the above considerations, we arrive at 
\begin{equation} \label{TrIntegralKernel}
\tr\left( T^{(k)}_{f,(\sigma_0,\sigma_1)}\circ \varphi^{(k)}\right) \approx \sum_{w} \int \mu_w(y)f(\varphi(y)) B_{\sigma_1}^{(k)}(\alpha_w^*(y),\alpha^*_w(\varphi(y))e^{ki\theta_w(y)} \ \Omega(y),
\end{equation}
where $\approx$ means equality up to addition of a function of $k$ which is $O\left(k^{-N}\right)$ for every $N>0.$ Recall $\{\tilde{b}_v\}$ from \eqref{expandingBergmankernel} and define for $u,v \geq 0$ the function  $f^w_{u,v} \in C^{\infty}(U_w)$ as follows
\begin{equation}
f^w_{u,v}=\mu_w\cdot f_u \circ \varphi \cdot \tilde{b}_v\circ R.
\end{equation} 
Applying \eqref{expandingBergmankernel} to the integrand of \eqref{TrIntegralKernel}, we obtain the following expansion
\begin{align}
\begin{split}
\\ \tr \left(Z^{(k)}(\varphi)\right)& = \sum_{u=0}^m \tr\left( T^{(k)}_{f_u,(\sigma_0,\sigma_1)}\circ \varphi^{(k)}\right)\tilde{k}^{-u}+O\left(k^{-(m+1)}\right)
\\&=k^{n_0} \sum_{u,v=0}^{N} \sum_{w}\left( \int_{U_w} f^w_{u,v} e^{kP_w} \ \Omega \right) k^{-v}\tilde{k}^{-u}
\\ & +O\left(k^{n_0-(N+1)}\right). \end{split}
\end{align}
Here $P_w$ is given by \eqref{Phase}. This proves the first part of the theorem by a simple power series substitution relating $k^{-1}$ to $\tilde{k}^{-1}$. 

 In order to prove the second half, we first make the following observation. Suppose $V\subset \C^n $ is an open neighborhood of the origin and $g$ a real analytic function on $V$, then there is a preferred almost analytical extension $\hat{g}$ to $V\times V$ which satisfies 
\begin{equation}
\hat{g}(v_1,v_2)=\overline{\hat{g}(v_2,v_1)},
\end{equation}
and the map $g \mapsto \hat{g}$ is $\R$-linear. Writing 
 \begin{equation}
 g(v)=\sum c_{a,b}v^{a}\overline{v}^b,
 \end{equation}
 near $0$ for $a,b \in \N^n,$ we take 
 \begin{equation}
 \hat{g}(v_1,v_2)=\sum c_{a,b}v_1^{a}\overline{v_2}^b.
 \end{equation}
 As $g$ is real-valed we have,
  \begin{equation} \label{c_{a,b}}
 c_{b,a}=\overline{c_{a,b}}.
 \end{equation}
 Hence $\hat{g}$ satisfies all of the desired properties.

Assume that the pre-quantum action is real analytic. Note that this entails that the real and imaginary parts of $P_w$ are both real analytic. The potential ambiguity in defining $P$ lies in the choice of the section $s=s_w,$ the choice of a real analytic extension $\tilde{\Phi}_s(y_1,y_2)$ of $\Phi_s,$ and the choice of $\theta_w$ gives a $2\pi i\Z$ ambiguity, which we shall ignore for now for notational convenience. Write $P_w=P_s,$ where $s$ is the chosen section. We claim that if we choose 
\begin{equation}
\tilde{\Phi}_s=\hat{\Phi}_s,
\end{equation} 
then $P_s$ will in fact be independent of $s.$  Any other choice of $s$ will be of the form 
\begin{equation} s'=e^gs,\end{equation} for some holomorphic function $g.$  With the obvious notation, we have that
\begin{gather}
\Phi_{s'}=-g-\overline{g}+\Phi_s, \label{ramram}
\ \ \ i\theta_{s'}=i\theta_{s}+2^{-1}( g-g\circ \varphi+\overline{g}\circ  \varphi-\overline{g}).
\end{gather}
From \eqref{ramram}, holomorphicity of $g$ and linearity of $r \mapsto \hat{r}$ we observe
\begin{equation}
\hat{\Phi}_{s'}(y_1,y_2)=-g(y_1)-\overline{g}(y_2)+\hat{\Phi}_s(y_1,y_2). 
\end{equation}
Combining these observations, we can make the following computation
\begin{align}
P^{\varphi}_{s'}(z)&= i\theta_{s'}(z)+\left(\hat{\Phi}_{s'}(z,\varphi(z))-2^{-1}\left( \Phi_{s'}(\varphi(z)) +\Phi_{s'}(z)\right) \right)
\\ &= i\theta_{s}(z)+\left(\hat{\Phi}_{s}(z,\varphi(z))-2^{-1}\left( \Phi_{s}(\varphi(z)) +\Phi_{s}(z)\right) \right)
\\ &+2^{-1}( g(z)-g(\varphi(z))+\overline{g}(\varphi(z))-\overline{g}(z))
\\&-g(z)-\overline{g}(\varphi(z))
\\ &+ 2^{-1}\left(g(z)+\overline{g}(z)+g(\varphi(z))+\overline{g}(\varphi(z)) \right)
\\ &=i\theta_{s}(z)+\left(\hat{\Phi}_{s}(z,\varphi(z))-2^{-1}\left( \Phi_{s}(\varphi(z)) +\Phi_{s}(z)\right) \right)
\\ &= P^{\varphi}_s(z).
\end{align}
This concludes the proof.
\end{proof}

In order to analyze further the expansion given in \eqref{lolExpansion of Trace}, we turn to the integrals \begin{equation}
I(w,n,k)=\int_{U_w}  e^{\tilde{k}P_w} \ \Omega^w_n.
\end{equation}
The large $k$ behavior of $I(w,n,k)$ localize to $U_w\cap M^{\varphi},$ as the real part of $P_w$ is strictly negative away from the fixed point set cf. \eqref{nonpositiveofchiawayfromdiag}. We prove that the set of fixed points of $\varphi$ correspond to stationary points, and we proceed to examine the Hessian of $P_w^{\varphi}$ at a fixed point.

 As the analysis of the phases are purely local, we may omit the indexes. We focus on a holomorphic coordinate chart $U$ centered at a fixed point $p,$ with Kähler coordinates $z$ satisfying
  \begin{equation} \label{KahlerCoordinates}
  -i\omega(p)=2^{-1} \sum_l \dv z_l \wedge \dv \overline{z}_l. 
 \end{equation} Let $y=(y_1,y_2)$ be the holomorphic coordinates on $U\times U,$ which are naturally induced by $z.$ For the rest of Subsection \ref{QIasOI}, we use the coordinates $z$ on $U,$ and the coordinates $y$ on $U\times U.$ 

  Recall that the construction of the phase $P$ function involves a choice of holomorphic frame 
\begin{equation}\label{s}
s: U \rightarrow L\setminus\{0\}.\end{equation}
 Define $\zeta=\zeta_s$ by writing the lift $\tilde{\varphi}$ with respect to the frames $s$ and $\varphi^*(s)$ as follows
\begin{equation}\label{zeta}
\tilde{\varphi}(s)(z)=\zeta(z) s(\varphi(z)).
\end{equation}
As $\zeta$ is non-vanishing, we can write (shrinking $U$ if necessary)  
\begin{equation}
\zeta=\exp(\lambda+i\theta),
\end{equation}
where $\theta$ and $\lambda$ are smooth real valued functions. In the  notation above, $\theta_w=\theta.$ As $\tilde{\varphi}$ is unitary we can write
\begin{equation}
\tilde{\varphi}(\exp(2^{-1}\Phi)s)=\exp(i\theta)\left(\exp(2^{-1}\Phi\circ \varphi)\varphi^*(s)\right).
\end{equation}
From this we conclude that
\begin{equation} \label{realpartoflift}
\lambda=2^{-1}(\Phi\circ \varphi-\Phi).
\end{equation}

We shall now prove that all fixed points of $\varphi$ are indeed stationary points of $P.$ Moreover, we can arrange that $\chi$ has a stationary point at $(p,p)$ by choosing the holomorphic section $s$ in \eqref{s} suitably. Write
$$I=I_{n_0\times n_0},$$
for the identity matrix of dimension $n_0.$
\begin{Lemma}\label{Prop1}
We have
\begin{equation} \label{fixedpoints=stationary}U \cap M^{\varphi}\subset \{ q \in U \mid \dv P_q=0\}.\end{equation}
Moreover, we can choose the holomorphic section $s$ in \eqref{s}  such that
\begin{gather} \begin{split} \label{dP=0} \dv \Phi_p=0, \ \ \ \dv \chi_{(p,p)}=0, \ \ \ \frac{\partial^2 \Phi}{\partial z^2}(p)=0, \ \ \ \frac{\partial^2 \Phi}{\partial z \partial \overline{z}}(p)=2^{-1}I.\end{split}\end{gather}\end{Lemma}

\begin{proof}
Let $q$ be a fixed point of $\varphi$ and let us show that $q$ is a stationary point by comparing the derivatives of $\chi \circ R$ with the derivatives of $i\theta.$  By differentiating \eqref{diagonal} and using \eqref{vanishatDIA}  we get
\begin{gather} \begin{split} \label{ext1}
\frac{\partial \Phi}{\partial z}= \frac{\partial \tilde{\Phi}}{\partial y_1}\circ \Delta,
\ \ \
\frac{\partial \Phi}{\partial \overline{z}}= \frac{\partial \tilde{\Phi}}{\partial \overline{y}_2}\circ \Delta.
 \end{split}
\end{gather}
By using \eqref{vanishatDIA} and \eqref{ext1} one obtains
\begin{align} \begin{split}\label{Together:)}
\frac{\partial \chi\circ R}{\partial z} (q)&=2^{-1}\left( \frac{\partial \Phi}{\partial z}+\frac{\partial \Phi}{\partial \overline{z}}\frac{\partial \overline{\varphi}}{\partial z}-\frac{\partial \Phi}{\partial z}\frac{\partial \varphi}{\partial z}\right)(q),
\\
\frac{\partial \chi\circ R}{\partial \overline{z}}(q)&= 2^{-1}\left(\frac{\partial \Phi}{\partial \overline{z}}\frac{\partial \overline{\varphi}}{\partial \overline{z}}- \frac{\partial \Phi}{\partial \overline{z}}-\frac{\partial \Phi}{\partial z}\frac{\partial \varphi}{\partial \overline{z}}\right)(q).
 \end{split}
\end{align}

In order to calculate the derivatives of $i\theta$ at $q,$ we shall first relate the derivatives of $\log(\zeta)$ to $\Phi$ using that $\tilde{\varphi}$ is parallel with respect to $\widetilde{\nabla},$ and then calculate the derivatives of $i\theta$ using the identity \eqref{realpartoflift}. Here $\widetilde{\nabla}$ is the connection on $\varphi^*(L) \otimes L^*$ naturally induced from $\nabla.$ Write 
\begin{equation}
\vartheta=\log(\zeta).
\end{equation} Define the $(1,0)$-form  $\psi$ by
 \begin{equation}\label{defpsi}
 \nabla s= \psi\otimes s.
 \end{equation} Write 
 \begin{equation}
 \widetilde{\nabla}(\varphi^*(s)\otimes s^*)=\kappa \otimes \varphi^*(s) \otimes s^*,
 \end{equation} where $\kappa$ is a $1$-form. We have 
 \begin{equation}\label{Beta}
 \kappa=\varphi^*(\psi)-\psi.
 \end{equation}
Thinking of $\tilde{\varphi}$ as a section of $ \varphi^*(L)\otimes L^*$ we can write
\begin{equation}
\tilde{\varphi}=\exp(\vartheta)  \varphi^*(s)\otimes s^*.
\end{equation}
By parallelity of $\tilde{\varphi}$ we get
\begin{equation}
0=\widetilde{\nabla} \tilde{\varphi} =\left( \dv \vartheta+ \kappa\right)\left(\exp(\vartheta)  \varphi^*(s) \otimes s^* \right),
\end{equation}
which is equivalent to
\begin{equation}\label{relatingvarthetapsi}
\dv \vartheta= \psi-\varphi^*(\psi).
\end{equation}
In particular
\begin{gather}
\begin{split} \label{derivativesofvartheta1}
\frac{\partial \vartheta}{\partial z_l} = \psi^l-\sum_j \psi^j\circ \varphi \ \frac{\partial \varphi^j}{\partial z_l},
\ \ \ \frac{\partial \vartheta}{\partial \overline{z}_l} = -\sum_j \psi^j\circ \varphi \ \frac{\partial \varphi^j}{\partial \overline{z}_l}.
\end{split}
\end{gather}
As $h,\nabla$ are compatible, we have 
\begin{align} \begin{split} \label{Phirelpsi}
\dv \Phi&= -\dv \left(\log(h(s,s)\right)
\\ &=-\psi-\overline{\psi}, \end{split}
\end{align}
from which we get
\begin{gather}\label{derivativesofPhi}\begin{split} 
\frac{\partial \Phi}{\partial z_l} = -\psi^l,
\ \ \ \frac{\partial \Phi}{\partial \overline{z}_l} = -\overline{\psi^l}.\end{split}
\end{gather}
Combining this with the expressions for the derivatives of $\vartheta$ given above, we get
\begin{gather}
\begin{split} \label{derivativesofvartheta2}
\frac{\partial \vartheta}{\partial z} =\left(\frac{\partial \Phi}{\partial z}\circ \varphi \right)\frac{\partial \varphi}{\partial z}-\frac{\partial \Phi}{\partial z},
\ \ \ \frac{\partial \vartheta}{\partial \overline{z}} &= \left(\frac{\partial \Phi}{\partial z}\circ \varphi \right)\frac{\partial \varphi}{\partial \overline{z}}.
\end{split}
\end{gather}
 Combining with  \eqref{realpartoflift} this gives 
\begin{align}\begin{split}
\frac{\partial i\theta}{\partial z} &=\frac{\partial \left(\vartheta-2^{-1}\left(\Phi\circ \varphi-\Phi\right)\right)}{\partial z}
\\ &=2^{-1}\left(\left(\frac{\partial \Phi}{\partial z}\circ \varphi \right)\frac{\partial \varphi}{\partial z}-\left(\frac{\partial \Phi}{\partial \overline{z}}\circ \varphi \right)\frac{\partial \overline{\varphi}}{\partial z}-\frac{\partial \Phi}{\partial z} \right).
\end{split}
\end{align}
As $\theta,\Phi$ are real functions, we obtain
\begin{align} \begin{split} \label{derivativesofitheta}
\frac{\partial i\theta}{\partial z} &=2^{-1}\left(\left(\frac{\partial \Phi}{\partial z}\circ \varphi \right)\frac{\partial \varphi}{\partial z}-\left(\frac{\partial \Phi}{\partial \overline{z}}\circ \varphi \right)\frac{\partial \overline{\varphi}}{\partial z}-\frac{\partial \Phi}{\partial z} \right),
\\\frac{\partial i\theta}{\partial \overline{z}} &=-2^{-1}\left(\left(\frac{\partial \Phi}{\partial \overline{z}}\circ \varphi \right)\frac{\partial \overline{\varphi}}{\partial \overline{z}}-\left(\frac{\partial \Phi}{\partial z}\circ \varphi \right)\frac{\partial \varphi}{\partial \overline{z}}-\frac{\partial \Phi}{\partial \overline{z}} \right).
\end{split}
\end{align}
Comparing \eqref{derivativesofitheta} with \eqref{Together:)} we see that $q$ is a stationary point.

We now proceed to prove \eqref{dP=0}.
 We see that if we perform the transformation $s \mapsto \exp(r)s,$ then the $(1,0)$-form $\psi$ transform as follows $ \psi \mapsto \psi+\partial r.$ Consider the holomorphic function $r$ given by
\begin{equation}
r=\sum_l \frac{\partial \Phi}{\partial z_l}(p)z_l+\sum_{l\leq j} \frac{\partial^2 \Phi}{\partial z_l \partial z_j}(p)z_lz_j.
\end{equation}
Replace $s$ by $\acute{s}=\exp(r)s,$ and observe that by \eqref{derivativesofPhi} we have
\begin{align} \begin{split} \label{dPhi=0}
\dv \acute{\Phi}_p&=0,
\\ \frac{\partial^2 \acute{\Phi}}{\partial z_l \partial z_j}(p)&=0.
\end{split}
\end{align}
We now turn to the fourth equation. we recall that as $\nabla$ is the Chern connection of $h,$ and as $L$ is a pre-quantum line bundle, we have
\begin{equation}
\overline{\partial}\partial \log(h(\acute{s}))_p= F_{\nabla}(p)=-i\omega(p)=2^{-1} \sum_l \dv z_l \wedge \dv \overline{z}_l.
\end{equation}
Thus we obtain the desired equation, which is seen to hold independently of the choice of section  
\begin{equation}
\frac{\partial^2 \acute{\Phi}}{\partial z \partial \overline{z}}(p)=2^{-1}I.
\end{equation}
We now turn to the remaining equation. By \eqref{ext1} we see that
\begin{gather}
\begin{split}
\frac{\partial \acute{\chi}}{\partial y}\circ \Delta  = \begin{pmatrix}
2^{-1}\frac{\partial \acute{\Phi}}{\partial z},& -2^{-1}\frac{\partial \acute{\Phi} }{\partial z}\end{pmatrix},
\ \ \ \frac{\partial \acute{\chi}}{\partial \overline{y}}\circ \Delta  = \begin{pmatrix}
-2^{-1}\frac{\partial \acute{\Phi}}{\partial \overline{z}},& 2^{-1}\frac{\partial \acute{\Phi} }{\partial \overline{z}}\end{pmatrix}.
\end{split}
\end{gather}
Therefore $\acute{\chi}$ has a stationary point at $(p,p)$ as $\acute{\Phi}$ has. \end{proof}

From now on, we will assume that the holomorphic section $s:U \rightarrow L\setminus\{0\}$ in \eqref{s} has been chosen as in Lemma \ref{Prop1}. In order to calculate $\text{Hess}(P)_p$ we make use of the following lemma.
\begin{Lemma}\label{Hessian}
Assume $ Q: U \rightarrow \C,$ is a smooth function defined on an open subset $U$ of $\C^m.$ Assume 
$
R: V \rightarrow U
$ is a smooth map defined on an open subset $V$ of $\C^l.$ If $R(z)$ is a stationary point of $Q$ then the Hessian  of $G$ at $z$ is given by
\begin{equation}
\text{Hess}(Q \circ R)= \dv R^t \text{Hess}(Q) \dv R.
\end{equation}
 \end{Lemma}

\begin{proof} This is a standard computation. \end{proof} The equations \eqref{dP=0} makes it significantly easier to analyze the Hessian of $P$ at a fixed point, together with Lemma \ref{Hessian}. We get a coordinate expression of the Hessian in terms of $\dv \varphi.$   \begin{Lemma}\label{theorem 1}  Write\begin{gather} \frac{\partial \varphi}{\partial z}(p)=H,\ \ \frac{\partial \varphi }{\partial \overline{z}}(p)=K.\end{gather}The Hessian of $P$ at $p$ with respect to $(\partial z, \partial \overline{z})$ is given as follows 
\begin{equation}\label{HessP2} \text{Hess}(P)_p=
4^{-1}\begin{pmatrix} I & 0\\ H & K\\ 0 & I\\ \overline{K} & \overline{H}\end{pmatrix}^t\begin{pmatrix}0 & 0 & -I & 2I\\ 0 & 0 & 0 & -I \\ -I & 0 & 0 & 0\\ 2I & -I & 0 & 0\end{pmatrix} \begin{pmatrix} I & 0\\ H & K\\ 0 & I\\ \overline{K} & \overline{H}\end{pmatrix}.\end{equation} \end{Lemma} 

In fact Lemma \ref{theorem 1} can be rephrased with the following formula valid at our fixed point $p$
\begin{equation}
\text{Hess}(P)= \dv R^t \text{Hess}(\chi) \dv R,
\end{equation}
where $R(z)=(z,\varphi(z)).$ 

\begin{proof}
We first show that the Hessian of $i\theta$ vanish at $p,$ from which we conclude that $\text{Hess}(P)_p=\text{Hess}(\chi\circ R)_p,$ and we  then derive a coordinate expression for $\text{Hess}(\chi\circ R)_p$ using Lemma \ref{Prop1} and Lemma \ref{Hessian}.

From \eqref{KahlerCoordinates} we see that the fact that $\dv \varphi_p$ is a symplectomorphism with respect to $\omega_p$ is equivalent to the following two equations
\begin{gather} \begin{split}\label{1a}
H^t\overline{H}-\overline{K}{}^tK=I, 
\ \ \ H^t\overline{K}-\overline{K}{}^tH =0. \end{split}
\end{gather}
Combining equation \eqref{derivativesofitheta} with equation \eqref{dP=0} and equation \eqref{1a} 
\begin{multline}\label{secondorderderivativesoftheta}
\begin{split}
\frac{\partial^2 i\theta}{\partial z^2} (p)&=2^{-1}\left(\frac{\partial \varphi}{\partial z}^t \left(\frac{\partial^2 \Phi}{\partial z^2}\circ \varphi \right)\frac{\partial \varphi}{\partial z} \right)(p)
\\ &+2^{-1}\left(\frac{\partial \overline{\varphi}}{\partial z}^t \left(\frac{\partial^2 \Phi}{\partial \overline{z} \partial z}\circ \varphi \right)\frac{\partial \varphi}{\partial z} \right)(p)
\\ &-2^{-1}\left(\frac{\partial \varphi}{\partial z}^t \left(\frac{\partial^2 \Phi}{\partial z \partial \overline{z}}\circ \varphi \right)\frac{\partial \overline{\varphi}}{\partial z} \right)(p)
\\ &-2^{-1}\left(\frac{\partial \overline{\varphi}}{\partial z}^t \left(\frac{\partial^2 \Phi}{ \partial \overline{z}^2}\circ \varphi \right)\frac{\partial \overline{\varphi}}{\partial z} \right)(p)
\\ &-2^{-1}\frac{\partial^2 \Phi}{\partial z^2}(p)
\\ &=4^{-1}\left( \overline{K}{}^tH-H^t \overline{K}\right)=0.
\end{split}
\end{multline}
Similarly one can use the equations \eqref{derivativesofitheta}, \eqref{dP=0} together with \eqref{1a} to show that
\begin{equation}
\frac{\partial^2 i\theta}{\partial \overline{z} \partial z }(p)=\frac{\partial^2 i\theta}{\partial z \partial \overline{z}}^t(p)=0.
\end{equation}
This concludes the first step and we obtain
\begin{equation}
\text{Hess}(P)_p= \text{Hess}(\chi\circ R)_p.
\end{equation}

 We claim that the Hessian of $\chi$ at $(p,p)$ with respect to the ordered basis $(\partial y_1, \partial y_2, \partial \overline{y}_1, \partial \overline{y}_2)$ is equal to the following matrix

\begin{equation}\label{hesschi} \text{Hess}(\chi)_{(p,p)}=
4^{-1}\begin{pmatrix}
0 & 0 & -I & 2I
\\ 0 & 0 & 0 & -I 
\\ -I & 0 & 0 & 0
\\ 2I & -I & 0 & 0
\end{pmatrix}.
\end{equation}

We shall calculate the Hessian of $\chi$ as four block matrices, and we start with $\frac{\partial^2 \chi }{\partial y^2}\circ \Delta(p).$ Differentiating equation \eqref{ext1} and using \eqref{vanishatDIA} we get
\begin{equation} \label{ext11}
\frac{\partial^2 \Phi}{\partial z^2}=\frac{\partial^2 \tilde{\Phi} }{\partial y_1^2}\circ \Delta.
\end{equation}
Combining equation \eqref{dP=0} and equation \eqref{ext11} we obtain
\begin{equation}
\frac{\partial^2\chi }{\partial y_1^2}\circ \Delta(p)=\frac{\partial^2 \tilde{\Phi} }{\partial y_1^2}\circ \Delta(p)-2^{-1}\frac{\partial^2 \Phi}{\partial z^2}(p)=0.
\end{equation}
Another two applications of \eqref{vanishatDIA} gives
\begin{equation}
\frac{\partial^2 \chi }{\partial y_1 \partial y_2}\circ \Delta(p)=\frac{\partial^2 \chi }{\partial y_2 \partial y_1}^t\circ \Delta(p)=0,
\end{equation}
and
\begin{equation}
\frac{\partial^2 \chi }{\partial y_2^2}\circ \Delta(p)=-2^{-1} \frac{\partial \Phi}{\partial z^2}(p)=0.
\end{equation}
Thus we see 
\begin{equation}
\frac{\partial^2 \chi }{\partial y^2}\circ \Delta(p)=\begin{pmatrix}
0 & 0
\\ 0 & 0
\end{pmatrix}.
\end{equation}
The computation of $\frac{\partial^2 \chi }{\partial y \partial \overline{y}}\circ \Delta(p)$ and $\frac{\partial^2 \chi }{\partial \overline{y}^2}\circ \Delta(p)$ is similar.  We can now appeal to Lemma \ref{Hessian} as
\begin{equation}
\dv R=\begin{pmatrix} I & 0
\\ H & K
\\ 0 & I
\\ \overline{K} & \overline{H}
\end{pmatrix}.
\end{equation} This concludes the proof. \end{proof}

We now address the question of non-degeneracy of the Hessian of $P$. Roughly Theorem \ref{NondegofHessPatrealcomplementstoFixedpointset} below shows that if $M^{\varphi}$ is cut out transversely, then our phase function $P$ is a Bott-Morse function.

\begin{theorem}\label{NondegofHessPatrealcomplementstoFixedpointset} Denote the symmetric bilinear form associated to $\text{Hess}(P)_p$ by $\Psi.$ We have
\begin{equation} \{ \xi \in T_p M_{\C} : \Psi(\xi,\eta)=0, \ \forall \eta \in T_pM_{\C} \}= \text{Ker}(\dv \varphi_p-I).\end{equation}Moreover, $\Psi$ is non-degenerate on any real complementary subspace to $\text{Ker}(\dv \varphi_p-I).$
\end{theorem}

\begin{proof} 
Let $\eta \in T_p M_{\C}$ be any vector not fixed by $\dv \varphi.$  We shall argue that 
\begin{equation} \label{denvigtigeligning}
\Psi(\eta,\overline{\eta})\not=0,
\end{equation}
which will imply that the symmetric bilinear form $\Psi$ is non-degenerate on any subspace which is complimentary to $\text{Ker}(\dv \varphi_p-I)$ and closed under conjugation.

 Write $\dv R_p(\eta)=(v',w',v'',w'').$ We have $\dv R_p(\overline{\eta})=(\overline{v''},\overline{w''},\overline{v'},\overline{w'}),$ as $\dv \varphi$ is the complex-linear extension of a real linear endomorphism of $T_p M.$  Let $(\emptyarg, \emptyarg)$ be the standard Hermitian metric on $\C^{n_0},$ given by $(u,y)=u^t\overline{y}$, where we recall that $n_0=\text{dim}_{\C}(M).$ According to Lemma \ref{theorem 1} we have
\begin{equation}\begin{split}
-4\Psi(\eta,\overline{\eta})=(v',v')+(w',w')-2(v',w')+(v'',v'')+(w'',w'')-2(w'',v'').
 \end{split}
\end{equation}
For any two vectors $\kappa,\mu$ the real part of $(\kappa,\kappa)+(\mu,\mu)-2(\kappa,\mu)$ is equal to $\lvert \kappa-\mu \rvert^2.$ Thus \eqref{denvigtigeligning} holds, as 
\begin{equation}
-4\text{Re}\left(\Psi(\eta,\overline{\eta})\right)=\lvert v'-w'\rvert^2+\lvert v''-w'' \rvert^2 \not=0.
\end{equation}

We finish the proof by proving the inclusion 
\begin{equation}
 \{ \xi \in T_p M_{\C} : \Psi(\xi,\eta)=0, \ \forall \eta \in T_pM_{\C} \}\supset \text{Ker}\left(\dv \varphi_p-I\right).
\end{equation}  To that end, assume $\xi$ is fixed by $\dv \varphi_p.$  Write $\xi'$ for the $(1,0)$ part of $\xi$ and write $\xi''$ for the $(0,1)$ part of $\xi.$ We compute 
 \begin{align}\begin{split}
 -4\Psi(\emptyarg, \xi)&=\begin{pmatrix}
  {\pi^{(1,0)}}^t & {\dv' \varphi_p}^t & {\pi^{(0,1)}}^t & {\dv'' \varphi_p}^t\end{pmatrix} \begin{pmatrix}
0 & 0 & I & -2I
\\ 0 & 0 & 0 & I 
\\ I& 0 & 0 & 0
\\ -2I & I & 0 & 0
\end{pmatrix}\begin{pmatrix}
\xi'
\\ \xi'
\\ \xi''
\\ \xi''
\end{pmatrix}
\\ &=\left(\pi^{(0,1)}- \dv'' \varphi_p\right)^t\xi' +\left(\dv' \varphi_p-\pi^{(1,0)}\right)^t\xi''. \end{split} \end{align}
That $\xi$ is fixed by $\dv \varphi_p$ is equivalent to
\begin{gather}
\begin{split}
\xi' = H\xi'+K\xi'',
\ \ \ \xi'' &= \overline{K}\xi'+ \overline{H}\xi''.
\end{split}
\end{gather} We have
\begin{gather}\begin{split}
\dv ' \varphi_p = \begin{pmatrix}
H & K
\end{pmatrix},
\ \ \
\dv '' \varphi_p = \begin{pmatrix}
\overline{K}  \overline{H}
\end{pmatrix},
\ \ \ \pi^{(1,0)} &= \begin{pmatrix}
I & 0
\end{pmatrix},
\ \ \ \pi^{(0,1)} = \begin{pmatrix}
0 & I
\end{pmatrix}. \end{split}
\end{gather}
 Thus we can write
\begin{multline}
\left(\pi^{(0,1)}- \dv'' \varphi_p\right)^t\xi' +\left(\dv' \varphi_p-\pi^{(1,0)}\right)^t\xi''= \begin{pmatrix}
-\overline{K}^t \xi'
\\ \xi'-\overline{H}{}^t\xi'\end{pmatrix}+ \begin{pmatrix}
H^t \xi''-\xi''
\\ K^t \xi''
\end{pmatrix}.
\\ = \begin{pmatrix}
\left(H^t\overline{K}-\overline{K}{}^tH \right)\xi'+\left(H^t\overline{H}-\overline{K}{}^tK-I\right)\xi''
\\ \left(I-\overline{H}{}^t H+K^t\overline{K}\right)\xi'+\left(K^t\overline{H}-\overline{H}{}^tK \right)\xi'' 
\end{pmatrix}=0. 
\end{multline}
The last equation follows from \eqref{1a}. \end{proof}

\subsection{Asymptotic expansions}\label{AECPROOFS}

Before proving the existence of asymptotic expansions, we shall introduce some notation. Assume that $M^{\varphi}$ is a smooth submanifold, or that the action is real analytic such that $M^{\varphi}$ is a real algebraic subvariety. For each component $Y \subset M^{\varphi}$ there exists $ \theta \in \R / 2\pi\Z,$ such that  for any $q \in Y,$ the linear endomorphism
\begin{equation}
\varphi^{(k)}_q: L^{k}_q \rightarrow L^k_q,
\end{equation}
is given by multiplication by $e^{ik\theta}.$ Let $\{
\theta_j  \in \R / 2\pi \Z \}$ be the set of arguments defined in this way.

We now recall the theorem of stationary phase approximation with parameters due to Hörmander \cite{H}. Let $U \subset \R^{n}\times \R^{m}$ be an open neighborhood of $(0,0),$ with coordinates $(u,v).$ Let $F \in C^{\infty}(U,\C).$  Assume that $F$ satisfies the following conditions
\begin{gather}\begin{split}\label{ConditionsForStationaryPhaseApproximation}
 \text{Re}(F)& \leq 0, \ \ \ \frac{\partial F}{\partial u}(0,0) =0, \ \ \ \text{Re}(F)(0,0)=0, \ \ \  \det\left( \frac{\partial^2 F}{\partial u^2}(0,0)\right)  \not=0.
\end{split}
\end{gather} 
Assume that $\phi \in C^{\infty}(U,\C)$ is of compact support concentrated  in a sufficiently small neighborhood of $(0,0).$ Then there exists a sequence of smooth differential operators $\{L_{F,j}\}_{ j \in \N},$ such that $L_{F,j}$ is  a differential operator in $u$ of order $2j,$ giving an asymptotic expansion
\begin{equation} \label{Est} \int \phi(u,v) e^{kF(u,v)} \dv u  \sim e^{kF^0(v)}\sqrt{\frac{2i\pi}{\det\left( k\frac{\partial^2 F}{\partial u^2}\right)^0(v)}} \sum_{j=0}^{\infty} k^{-j} L_{F,j}(\phi)^0(v),
\end{equation}
where for a function $G$ of $(u,v),$ the notation $G^0(v)$ stands for a function of $v$ only which is in the same residue class modulo the ideal generated by $\frac{\partial F}{\partial u_i},i=1,...,n.$ This means that there are smooth functions $q_j$ defined in an open neighborhood of $(0,0)$ such that on that neighborhood we have
 \begin{equation} \label{Preparation1'}
 G(u,v)= \sum_{j=1}^{n} q_j(u,v)\frac{\partial F }{\partial u_j}(u,v)+G^0(v).
 \end{equation}

We now apply this theorem to $\tr(Z^{(k)}(\varphi)).$ Recall the notion of a transversely cut out fixed point set, as defined in the introduction to this paper.

\begin{theorem} \label{Theorem2}
Assume $M^{\varphi}$ is cut out transversely. For each $\theta_j,$ define $2m_j$ as the maximal integer occuring as the dimension of a component of $M^{\varphi}$ on which $\tilde{\varphi}$ is given by multiplication by $\exp(i\theta_j).$ There exists a sequence of differential forms $\Omega_{\alpha}^j,$ giving an asymptotic expansion
\begin{equation} \label{mastereq}
 \tr \left(Z^{(k)}(\varphi)\right)\sim \sum_{j} e^{i \tilde{k} \theta_j} \tilde{k}^{m_j} \sum_{\alpha=0}^{\infty} \tilde{k}^{-\frac{\alpha}{2}} \int_{M^{\varphi}} \Omega^j_{\alpha}. 
\end{equation}
\end{theorem}
\begin{proof}
For notational convenience we shall assume that all the components of $M^{\varphi}$ are of the same dimension $2m_0$. Let $e_0$ be the co-dimension of $M^{\varphi},$ i.e. $2n_0-2m_0=e_0.$ Choose each holomorphic trivialization $s_w: U_w \rightarrow L\setminus \{0\}$ appearing in the covering of $M^{\varphi}$ given in \eqref{U_w} such that each $U_w$ admits smooth real coordinates $(u_w,v_w): U_w \rightarrow \R^{2n_0}$ that satisfies
\begin{equation}
\{u_w=0\} = U_w \cap M^{\varphi}.
\end{equation}
With notation as in Theorem \ref{QIasOI}, define $g^w_n\in C^{\infty}(U_w)$ by
\begin{equation}
\Omega^w_n=g^w_n \dv u_w \wedge \dv v_w.
\end{equation} We can now rewrite the expansion in Theorem \ref{QIasOI} as 
\begin{equation} \label{expaexpa}
 Z^{k}(\varphi)=  \tilde{k}^{n_0} \sum_{n=0}^M \sum_{w} \tilde{k}^{-n} \int \left( \int g^w_n(u_w,v_w) e^{\tilde{k}P_w(u_w,v_w)} \ \dv u_w \right) \dv v_w+ O\left(k^{n_0-(M+1)}\right).
\end{equation}
As the real part of $\chi$ is negative away from the diagonal we see that Theorem \ref{NondegofHessPatrealcomplementstoFixedpointset} implies that the phase function $P_w$ satisfy condition \eqref{ConditionsForStationaryPhaseApproximation}. We now wish to apply stationary phase approximation to the double integrals
\begin{equation}
I(n,w,\tilde{k})=\int \left( \int g^w_n(u_w,v_w) e^{\tilde{k}P_w(u_w,v_w)} \ \dv u_w \right) \dv v_w.
\end{equation}
 We shall use freely that a function of $k$  or $\tilde{k}$ is $O\left(k^{-M}\right)$ if and only if it is $O\left(\tilde{k}^{-M}\right).$  Let $\mathcal{I}$ be the ideal generated by
 $\frac{\partial P}{\partial u^1_w},...,\frac{\partial P}{\partial u^{e_0}_w}.$ With notation as above, we observe that for any smooth function $G(u_w,v_w),$ we have \begin{equation}
G^0(v_w)=G(0,v_w),
\end{equation} as 
\begin{equation}
\frac{\partial P}{\partial u_w}(0,v_w)=0,
\end{equation} by Lemma \ref{Prop1} and
\begin{equation}
G(u_w,v_w)-G^0(v_w) \in \mathcal{I}.
\end{equation} Thus by H\"{o}rmander's theorem of stationary phase approximation with parameters we get the following expansion
\begin{align} \label{jadadada}
&\int   g^w_n(u_w,v_w) e^{\tilde{k}P_w(u_w,v_w)}  \dv u_w+O\left(\tilde{k}^{-(N+1+e_0/2)}\right)
\\ &= e^{\tilde{k}P_w(0,v_w)}\tilde{k}^{-2^{-1}e_0}\sqrt{\frac{2i\pi}{\det\left(\frac{\partial^2 P_w}{\partial u_w^2}(0,v_w)\right)}} \sum_{r=0}^N \tilde{k}^{-r} L_{P_w,r}( g^w_n)(0,v_w),
\end{align}
where $L_{P,r}$ is the differential operator of order $2r$ as discussed above. From Lemma \ref{Prop1} we easily deduce that
\begin{equation} \label{lambadamba}
P_w(0,v_w)=i\theta_j,
\end{equation}
for some $j.$ Define $\Omega^w_{n,r}$ by
\begin{align}
\Omega^w_{n,r}(v_w)= \sqrt{\frac{2i\pi}{\det\left(\frac{\partial^2 P_w}{\partial u_w^2}(0,v_w)\right)}} L_{P_w,r}( g^w_n)(0,v_w)\ \dv v_w.
\end{align}
This form has compact support inside $U_l\cap M^{\varphi}.$ Using  \eqref{lambadamba} and \eqref{jadadada} we have the following estimate
\begin{equation} \label{jadadada2}
 I(n,w,\tilde{k})= e^{i\tilde{k}\theta_j}\tilde{k}^{-2^{-1}e_0} \sum_{r=0}^{N} \tilde{k}^{-r} \int_{M^{\varphi}} \Omega^w_{n,r}+O\left(k^{-(N+1+e_0/2)}\right).
\end{equation}
For each $\theta_j$ let $W_j$ be the set of $w'$ such that $U_{w'} \cap M^{\varphi}$ lies in a component of $M^{\varphi}$ on which $\tilde{\varphi}$ is given by multiplication by $\exp(i\theta_j).$
Define.
\begin{equation}
\Omega^j_{n,r}=\sum_{w \in W_j}  \Omega^w_{n,r}.
\end{equation}
Observe that
\begin{equation}
n_0-\frac{1}{2}e_0=m_0.
\end{equation}
Using \eqref{jadadada2}, we can rewrite \eqref{expaexpa} as follows
\begin{align} \label{crucial estimate} \begin{split}
\tr \left(Z^{(k)}(\varphi)\right)&=\tilde{k}^{m_0}\sum_{j} e^{ik\theta_j} \sum_{n,r=0}^{N} \tilde{k}^{-n-r} \int_{M^{\varphi}} \Omega^j_{n,r} +O\left(k^{m_0-(N+1)}\right) . \end{split}
\end{align}
Defining
\begin{equation}
\Omega^j_{\alpha}= \sum_{n+r=\alpha} \Omega^j_{n,r},
\end{equation}
we see that the theorem holds. \end{proof}

We are interested in the situation where the phase have more general degenerate stationary points. The most important theorems in this situation are due to Malgrange \cite{Malgrange}. The theorems we present below are also given a thorough treatment in the monograph \cite{Arnold}. Let $f \in C^{\infty}(\mathbb{R}^n,\mathbb{R}).$ Assume $f$ is real analytic near a stationary point $p.$ Assume that $\phi$ is a smooth function of compact support contained in a small neighborhood $D$ of $p.$ If $D$ is sufficiently small, we then have an asymptotic expansion
\begin{equation} \label{purelyIma}
\int _{\R^n} e^{kif(x)} \phi(x) \ \dv x \sim e^{kif(p)}\sum_{\alpha} \sum_{\beta=0}^{n-1} c_{\alpha,\beta} k^{\alpha}\log(k)^\beta,
\end{equation}
where $\alpha$ runs through a finite set of arithmetic progressions of negative rational numbers depending on $f.$ If $p$ is a maximum of $f,$ and $D$ is sufficiently small, we also have an asymptotic expansion 
\begin{equation} \label{Laplace}
\int _{\R^n} e^{kf(x)} \phi(x) \ \dv x \sim e^{kf(p)}\sum_{\alpha} \sum_{\beta=0}^{n-1} b_{\alpha,\beta} k^{\alpha}\log(k)^\beta,
\end{equation}
where $\alpha$ runs through a finite set of arithmetic progressions of negative rational numbers depending on $f.$ There are distributions $C_{\alpha,\beta}^f,B_{\alpha,\beta}^f$ supported at the level set $\{x:f(x)=f(p)\}$ such that
\begin{equation}
C_{\alpha,\beta}^f(\phi)=c_{\alpha,\beta}, \ B_{\alpha,\beta}^f(\phi)=b_{\alpha,\beta}.
\end{equation}
These distributions are of finite order.


We now move on to the holomorphic case. Assume that $f$ is a holomorphic function defined on an open subset $U \subset \mathbb{C}^n.$ Assume $f$ has an isolated stationary point at $p \in U.$ Assume $\Lambda \subset U$ is an integration chain of real dimension $n,$ and $p \in \Lambda.$ Let $\Omega$ be a holomorphic $n$ form defined on $U.$ Assume that $\text{Re}(f)_{\mid \partial \Lambda} < \text{Re}(f)(p).$ If $U$ is sufficiently small, we have an asymptotic expansion
 \begin{equation} \label{HolomorphicLaplace}
\int_{\Lambda} e^{kf} \ \Omega \sim e^{kf(p)} \sum_{\alpha,\beta} c_{\alpha,\beta} k^\alpha \log(k)^\beta.
 \end{equation}
 where $\alpha$ runs through a finite set of arithmetic progressions of negative rational numbers depending on $f,$ and $\beta$ runs through a finite subset of $\mathbb{N}$ depending on $f.$ We shall briefly illustrate the main ideas. We can assume without loss of generality that $f(p)=0.$ Using the Milnor fibration of $f$ at $p,$ we can deform the chain $\Lambda$ into a Picard-Lefschetz thimble which can be foliated by so-called vanishing cycles $\gamma(t)$ indexed by a real parameter $t<0.$ These are of real dimension $n-1$ and satisfy $f_{\mid \gamma(t)}=t.$ We have
 $$\int_{\Lambda} e^{kf} \omega \approx \int_{-\infty}^{0} e^{tk} \int_{ \gamma(t)} \frac{\Omega}{\dv f} \ \dv t,$$
where $\approx$ means equality up to addition of a function of $k$ which is $O\left(k^{-m}\right)$ for every $m \in \mathbb{N}.$ One can show that the so-called Gelfand-Leray transform $F(t)$ given by $$F(t)=\int_{ \gamma(t)} \frac{\Omega}{\dv f},$$
admits a multivalued analytic extension in $t.$ For $t$ near $0,$ there exists an expansion of $F(t)$ in terms of powers of $t$ and powers of $\log(t).$ Integrating this expansion against $e^{-kt}$ gives \eqref{HolomorphicLaplace}. The Gelfand-Leray transform $F(t)$ is closely related to the monodromy operator $U_f$ of $f$ at $p,$ and a fortiori so is the expansion \eqref{HolomorphicLaplace}. We recall that the monodromy operator is defined as an anti homomorphism $U_f : \pi_1(\mathbb{C}\setminus X,z_0) \rightarrow \text{Aut}\left( H_*(f^{-1}(z_0),\mathbb{Z})\right)$ where $X$ is the discriminant of $f,$ and $z_0$ is a regular value. The action is given by deforming a cycle continuously along level sets of $f,$ by lifting via $f$ a loop $\tilde{\gamma}$ in $\mathbb{C}\setminus X.$ An exponent $\alpha$ in \eqref{HolomorphicLaplace} exponentiates to an eigenvalue $\exp(-2\pi i\alpha)$ of $U_f,$ and $c_{\alpha,\beta}=0$ if $\beta-1$ is bigger than the dimension of any Jordan block of $U_f$ associated to the eigenvalue $\exp(-2\pi i\alpha).$ For more details, see \cite{Arnold}. We remark that these holomorphic oscillatory integrals have also  been thoroughly studied from a resurgence viewpoint by several authors including Dingle, Pham, Howls, Berry-Howls, and Howls-Delabaere \cite{Dingle73,Pham83,BerryHowls90,BerryHowls91,Howls97,DelabaereHowls02}, and the starting point is to consider what happens when one allows the parameter $k$ to become complex valued. 


 We now give the proof of Theorem \ref{AnalyticalResult2}, which allows us to go beyond the case where $M^{\varphi}$ is transversely cut out.

\begin{proof}[Proof of Theorem \ref{AnalyticalResult2}]

Let $I(k)=\int_{\R^n} e^{kf(x)} \phi(x) \ \dv x.$ Without loss of generality, we can assume that $p$ is the origin, and that $f(0)=0$. We shall start by reducing the first case to the second case. By using the theorem of stationary phase approximation with parameters we show that $I(k)$ has an asymptotic expansion in terms of products of negative powers of $k$ and $1$-dimensional oscillatory integrals $I_j(k),$ all of which have the the same phase $f^0$ whose imaginary and real parts are real analytic functions. Assume that $U$ is a coordinate neighborhood centered at the fixed point. We can apply a linear transformation in order to obtain coordinates $(x,y)$ on $U \subset \R^{n-1}\times \R,$ such that $\text{Hess}(f)_0$ is non-degenerate on $T_0 \R^{n-1}=\text{Span}(\partial x_1,..., \partial x_{n-1}).$ It follows from our assumptions and Theorem \ref{NondegofHessPatrealcomplementstoFixedpointset}, that $f$ satisfies the conditions for stationary phase approximation with parameters. With the notation as above, write $g(y)=f^0(y).$ Then there are smooth functions $q_j(x,y)$ such that
 \begin{equation} \label{Preparation1'}
 f(x,y)= \sum_{j=1}^{n-1} q_j(x,y)\frac{\partial f }{\partial x_j}(x,y)+g(y).
 \end{equation}
We wish to argue that we can arrange that the functions $\text{Im}(g)$ and $\text{Re}(g)$ are both real analytic, and that $\text{Re}(g)\leq 0.$ In Proposition $7.7.13$ in \cite{H}, Hörmander  proves that the expansion \eqref{Est} is valid for any function $g'(y)$ for which there exists functions $q'_1(x,y),...,q'_{n-1}(x,y)$ such that \eqref{Preparation1'} holds near $0,$ with $g$ replaced by $g',$ and $q_j$ replaced by $q'_j.$ Hence it will suffice to prove the following assertion: Assume that the real and imaginary parts of a function $s(x,y)$ are both real analytic, and that $f_1,....,f_{n-1}$ are functions whose real and imaginary parts are both real analytic and 
\begin{equation}
\begin{cases}
& f_j(0,0)=0, \ \text{for} \ j=1,...,n-1,
\\ & \det\left( \frac{f_j}{\partial x_r}(0,0) \right) \not=0.
\end{cases} 
\end{equation}
Then there are functions $h(y),q_j(x,y),j=1,...,n-1,$ with real analytic real and imaginary parts, such that in a neighborhood of $0$ we have
\begin{equation} \label{Preparation}
 s(x,y)= \sum_{j=1}^{n-1} q_j(x,y)f_j(x,y)+h(y).
 \end{equation} We see that we can arrange that $\text{Im}(g),\text{Re}(g)$ are real analytic by using this assertion in the case where $s=f$ and $f_j=\partial f / \partial x_j.$ 

 We now prove the assertion. Shrink the domain to assume that the following power series expansions are valid near $0$
 \begin{gather}
 f_j(x,y)= \sum c^{j}_{a,b}x^ay^b, \ \ s(x,y)= \sum c_{a,b}x^ay^b, 
 \end{gather}
 where $a$ ranges over $\N^{n-1}$ and $b$ range over $\N.$  Let $U \subset \R^n$ be a small enough domain such that these series are absolutely convergent. Consider the domain $U_1+iU_2 \subset  \C^n $ with coordinates $(z,w)$ with $z=x_1+ix_2$ and $w=y_1+iy_2.$ We can extend $f_j$ and  $s$ to holomorphic functions $\check{f}_j$ and $\check{s}$ on $U_1+iU_2$ by
 \begin{gather}
 \check{f}_j(z,w)= \sum c^{j}_{a,b}z^aw^b, \ \ \check{s}(z,w)= \sum c_{a,b}z^aw^b.
 \end{gather}
Observe that
\begin{equation}
\frac{\partial \check{f}_j }{\partial z_j}(x_1,y_1)= \frac{\partial f_j}{\partial x_j}(x_1,y_1).
\end{equation}
Therefore it will suffice to prove the holomorphic version of the assertion. This can be proven using the Weierstrass preparation theorem, analogously to how the proof of Theorem $7.5.7$ in \cite{H} relies on Malgrange's preparation theorem, which is Theorem $7.5.6$ in \cite{H}.  

 Thus we obtain an asymptotic expansion, such that for all $M\in \mathbb{N}$ we have
 \begin{align}
 & \int_{\R^{n-1}} e^{kf(x,y)} \phi(x,y) \ \dv x
 \\ & =e^{kg(y)}\sqrt{\frac{2i\pi}{\det\left( k\frac{\partial^2 f}{\partial x^2}\right)^0(y)}} \sum_{j=0}^{M} k^{-j} L_{f,j}(\phi)^0(y)+O\left( k^{-M-1-\frac{n-1}{2}}\right),
\end{align}
 where $\text{Im}(g),\text{Re}(g)$ are real analytic. Observe that as $\int_{\R^{n-1}} e^{kf(x,y)} \phi(x,y) \ \dv x$ is bounded as a function of $k,$ we see that it is also true that $\text{Re}(g) \leq 0.$ As the estimate is uniform for small $y,$ we have 
\begin{equation}
\label{sstep1.1}
 I(k)=\int_{\R}e^{kg(y)}\sqrt{\frac{2i\pi}{\det\left( k\frac{\partial^2 f}{\partial x^2}\right)^0(y)}} \sum_{j=0}^{M} k^{-j} L_{f,j}(\phi)^0(y)  \ \dv y  +O\left( k^{-M-1-\frac{n-1}{2}}\right).
\end{equation} 
Introduce
\begin{equation} \label{sstep1.2}
I_j(k)= \int_{\R} e^{kg(y)}\sqrt{\frac{2i\pi}{\det\left(\frac{\partial^2 f}{\partial x^2}\right)^0(y)}}L_{f,j}(\phi)^0(y) \ \dv y,
\end{equation}
and rewrite \eqref{sstep1.1} as
\begin{equation} \label{sstep1.3}
I(k)=k^{-\frac{n-1}{2}}\sum_{j=0}^{M}k^{-j}I_j(k)+O\left( k^{-M-1-\frac{n-1}{2}} \right). 
\end{equation} 
Without loss of generality, we can assume that
\begin{equation} \label{stationarypointofp}
\frac{\dv g}{\dv y}(0)=0.
\end{equation}
Otherwise the conclusion of the theorem is true with all coefficients being $c_{\alpha}(\phi)=0.$ From \eqref{Preparation} and $\dv f_0=0$ we also deduce
\begin{equation}
g(0)=f(0,0),
\end{equation}
which implies that $\text{Re}(g)(y)$ has a maximum at $y=0.$ Write $\text{Re}(g)(y)=a_1(y).$ As we are in the one-dimensional case, a real analytic function is either constant or has isolated zeroes and isolated stationary points. Thus, by appealing to the expansions  \eqref{purelyIma} and \eqref{Laplace}, we can assume without loss of generality that the holomorphic extension $\check{g}(y)$ has an isolated stationary point at $y=0,$ and $a_1(y)$ has an isolated zero at $y=0.$ Hence the existence of an expansion can be reduced to the second case. We shall return to the form of the expansion, after having proven its existence.

We now deal with the second case. Write
\begin{equation}
\int_U e^{k f(x)} \phi(x) \  \dv x= \int_U e^{kf(x)} T(\phi)(x) \ \dv x  +\int_U e^{kf(x)} R(\phi)(x) \ \dv x 
\end{equation} 
where $T(\phi)(x)$ is a Taylor polynomial of $\phi$ at $x=0$ of very high degree, such that $R(\phi)(x)$ vanishes to very high order at $x=0,$ say to order $m.$ Let $a=\text{Re}(f).$ By the Cauchy-Schwartz inequality we have
\begin{equation} \label{kolo}
\left\lvert \int_U e^{kf(x)} R(\phi)(x) \ \dv x \right\rvert  \leq \int_U \lvert e^{kf(x)} R(\phi)(x) \rvert \ \dv x  \leq \left(\text{Vol}(U)\int_U e^{2k a(x)} (R(\phi)(x))^2 \ \dv x \right)^{\frac{1}{2}}.
\end{equation}
We can expand $J_1(k)=\int_U e^{2k a(x)} (R(\phi)(x))^2 \ \dv x$ by using Malgrange's theorem for Laplace integrals. The coefficients of that expansion will be of the form $b_{\alpha,\beta}=B_{\alpha,\beta}^{2a}(R(\phi)),$ where $B_{\alpha,\beta}^{2a}$ is a distribution of finite order supported at the level set $\{x:a(x)=0\}.$ By assumption this level set meets $\text{supp}(\phi)$ only at the origin, at which $R(\phi)(x)$ vanishes to a high order $m.$ It follows that for any $M\in \mathbb{N}$ we can choose $m$ large enough to ensure that \begin{equation} \label{dominatingR} \int_U e^{kf(x)} R(\phi)(x) \ \dv x =O\left(k^{-M}\right).\end{equation}  We can expand $J_2(k)=\int_U e^{k\check{f}(x)} T(\phi)(x) \ \dv x,$ by using Malgrange's theorem for oscillatory integrals with holomorphic phase \eqref{HolomorphicLaplace}. Here we use that we can think of $U$ as a real $n$ chain inside $\mathbb{C}^n,$ and $T(\phi)(x) \ \dv x_1 \cdots \dv x_n$ can be seen as a holomorphic $n$ form on $\mathbb{C}^n.$ Combining the expansion of $J_2(k)$ with the estimate \eqref{dominatingR} gives a partial expansion of $I(k).$ Continuing this way, by Taylor expandanding $\phi$ to all orders, we get a full asymptotic expansion of $I(k)$ of the form \eqref{HolomorphicLaplace}.

  We now return to the form of the expansion in the first case. We appeal to Example $11.1.3$ in \cite{Arnold}, where the case of an oscillatory line integral with holomorphic phase is treated. Indeed, this case reduces to the case $g(y)=y^{m},$ and it is proven that the asymptotic expansion is of the form (for holommorphic $q(z)$ and over an appropriate contour through $0$)
  $$ \int e^{ky^{m}} q(y) \ \dv y \sim \sum_{j=1}^{\infty} c_j(q) k^{-\frac{j}{m}}.$$ This finishes the proof \end{proof}

Moreover; the proof of Theorem \ref{AnalyticalResult2} provides insight on stationary phase approximation with parameters.

\begin{corollary} \label{corollary} Assume that $F:\R^n \times \R^m \rightarrow \C$ is a complex valued function which satisfy \eqref{ConditionsForStationaryPhaseApproximation}. If $F$ is real analytic near $(0,0)$, one can choose the function $F^0$ in \eqref{Est} to be real analytic near $0.$ If $F$ is real and real analytic, then $F^0$ can be chosen to be real and real analytic as well.  
\end{corollary}
\begin{proof}
This was proven in the first step of the proof of Theorem \ref{AnalyticalResult2} when $F$ is complex valued, and the last statement follows from the Weierstrass preparation theorem for real analytic functions. 
\end{proof}

Corollary \ref{corollary} can be used to provide a bound for powers of $\log(k)$ appearing in the asymptotic expansions \eqref{purelyIma} and \eqref{Laplace}.

\begin{corollary} \label{corollary2}
Let $f \in C^{\infty}(\mathbb{R}^n,\mathbb{R}).$ Assume $f$ is real analytic near a stationary point $p.$ Assume that $\phi$ is a smooth function of compact support contained in a small neighborhood $D$ of $p.$ Assume that $\text{Hess}(f)_p$ is non-degenerate on a subspace of dimension $m<n$ and let $q=n-m-1.$ If $D$ is sufficiently small,  we then have an asymptotic expansion
\begin{equation}
\int _{\R^n} e^{kif(x)} \phi(x) \ \dv x \sim e^{kif(p)}\sum_{\alpha} \sum_{\beta=0}^{q} c_{\alpha,\beta} k^{\alpha}\log(k)^\beta,
\end{equation}
where $\alpha$ runs through a finite set of arithmetic progressions of negative rational numbers depending on $f.$ If $p$ is a maximum of $f,$ we also have an asymptotic expansion 
\begin{equation}
\int _{\R^n} e^{kf(x)} \phi(x) \ \dv x \sim e^{kf(p)}\sum_{\alpha} \sum_{\beta=0}^{q} b_{\alpha,\beta} k^{\alpha}\log(k)^\beta,
\end{equation}
where $\alpha$ runs through a finite set of arithmetic progressions of negative rational numbers depending on $f.$
\end{corollary}

\begin{proof} This is a straightforward application of Corollary \ref{corollary} and the theorems of Malgrange on asymptotic expansions of oscillatory integrals presented above in \eqref{purelyIma} and \eqref{Laplace}. \end{proof}

We now use Theorem \ref{AnalyticalResult2} to prove our main theorem. For each $\theta_j$ let $2m_j$ be the maximal value of $\text{dim}(\text{Ker}(\dv \varphi_z-I))$ among $z \in M^{\varphi}$ with $\tilde{\varphi}_z$ given by multiplication with $\exp(2\pi i\theta_j).$

\begin{theorem} \label{Main Theorem2} Assume the pre-quantum action is real analytic. If all $z \in M^{\varphi}$ satisfy one of the following three conditions \begin{itemize}  
\item  $z$ is a smooth point with $T_z M^{\varphi}=\text{Ker}(\dv \varphi_z-I),$
\item   $\text{dim}(\text{Ker}(\dv \varphi_z -I)) \leq 1,$ or
\item  $z$ is an isolated stationary point of the germ of the holomorphic extension $\check{P}^{\varphi}$ at $z,$
\end{itemize}   we then have an asymptotic expansion of the form \[
\tr \left(Z^{(k)}(\varphi)\right) \sim  \sum_j e^{\tilde{k}i\theta_j}\tilde{k}^{n_j}\sum_{\alpha \in A_j, \beta \in B_j} c_{\alpha,\beta} \tilde{k}^{\alpha} \log(\tilde{k})^{\beta}.
 \]
  Here $n_j \in \mathbb{Q}_{\geq 0},$ each $A_j\subset \mathbb{Q}_{\leq 0}$ is a union of finitely many arithmetic progressions, and each $B_j \subset \mathbb{N}$ is finite. If for all $z \in M^{\varphi}$ the first or second condition holds, then $B_j=\{0\}$ and $n_j=m_j$ for all $j.$ \end{theorem}

\begin{proof}

By Theorem \ref{QIasOI} we have \begin{equation}\label{Expansion of Trace}\begin{split}
\tr \left(Z^{(k)}(\varphi)\right) =\tilde{k}^{n_0} \sum_{n=0}^{N} \sum_{w }\left( \int_{U_w} e^{\tilde{k}P^{\varphi}} \ \Omega^w_n \right) \tilde{k}^{-n} +O\left(k^{n_0-(N+1)}\right). \end{split}
\end{equation} 
Hence it will be enough to show the existence of an asymptotic expansion of integrals of the following form 
\[
I(\tilde{k})=\int_U e^{\tilde{k}P^{\varphi}(x)} \phi(x) \ \dv x,
\]
where $U$ is an open neighborhood of a fixed point $z$, and $\phi$ is a smooth function of compact support concentrated in a small coordinate ball centered at $z$.  The case where $z$ is a smooth point of $M^{\varphi}$  with $T_z M^{\varphi}=\text{Ker}(\dv \varphi_z-I)$ is dealt with in Theorem \ref{Theorem2}. The case with $\text{dim}\text{Ker}(\dv\varphi_z-I) \leq 1,$ can be dealt with by appealing to the first part of Theorem \ref{AnalyticalResult2} and Theorem \ref{NondegofHessPatrealcomplementstoFixedpointset}. The case where $\check{P}^\varphi$ has an isolated stationary case is dealt with by appealing to the second part of Thereom \ref{AnalyticalResult2}. Here we use that any maximum of $\text{Re}(P^{\varphi})$ is a fixed point by \eqref{nonpositiveofchiawayfromdiag} and therefore a stationary point of $P^{\varphi}$ by Lemma \ref{Prop1}, and a fortiori a stationary point of $\check{P}^{\varphi}.$
\end{proof}

Before moving on to TQFT, let us briefly compare Theorem \ref{AnalyticalResult2} with the theorems of Malgrange.
\begin{remark} \label{Comparison}
The expansions \eqref{purelyIma} and \eqref{Laplace}  provided by the results of Malgrange, are valid for $f \in C^{\infty}(\mathbb{R}^n,\mathbb{R})$ which is real analytic near a stationary point or a maximum. Theorem \ref{AnalyticalResult2} deal with $f \in C^{\infty}(\mathbb{R}^n,\mathbb{C}),$ with real analytic real and imaginary part near a stationary point, which is also a maximum of the real part. For the first part of Theorem \ref{AnalyticalResult2} we need to impose a condition on the Hessian. For the second part of Theorem \ref{AnalyticalResult2} we can not appeal directly to the case with holomorphic phase \eqref{HolomorphicLaplace}, as we are not assuming that the amplitude form $\phi \ \dv x$ extends to a holomorphic form on $\mathbb{C}^n,$ which is an assumption in Malgranges result.
\end{remark}

\begin{remark} \label{MMM} For the integer $m$ appearing in \eqref{ourresult}, $m-1$ is the vanishing order of $f^0(y)-f^0(p)$ at $y=p,$ where $f^0(y)$ is the function given by Hörmander's stationary phase with parameters. Observe that if $n=1,$ the condition on the Hessian is vacuously true, and $f^0=f,$ so $m-1$ is the vanishing order of $f(y)-f(p)$ at $p.$ We remark that the leading order term of the expansion \eqref{ourresult} is in agreement with what is to be expected following the work of Varchenko \cite{Varchenko} and Vasil'ev \cite{VaVassiliev}. See also chapter $11$ of \cite{Arnold}.
\end{remark}

\section{TQFT and quantization of moduli spaces} \label{TQFT}

In these sections, we shall apply the results obtained above to quantum representations arising in relation to topological quantum field theory.

\subsection{The co-prime case} 

Recall the setup of the co-prime case from the introduction where ${\mathcal M}$ is the moduli space of flat $\rm{SU}(n)$-connections on the complement of a point $P$ on $\Sigma$, around which the holonomy is a given generator $\delta$ of the center of $\rm{SU}(n)$. It is shown in \cite{AJHM} that the setup of the co-prime case fit together to form a pre-quantum action where the symplectic manifold is $\mathcal M$ and $\T$ is Teichmüller space of $\Sigma$ and $\Gamma$ is the mapping class group of $(\Sigma,P,v_P),$ where $v_P$ is a tangent direction at $P$ . 

This is seen to be a real analytic action as follows. Consider the corresponding character variety $X$ in this co-prime case for the complexified group ${\rm SL}(n,{\mathbb C})$, which carries a natural complex structure coming from ${\rm SL}(n,{\mathbb C})$. Once we choose a point $\sigma\in \T$ we get a holomorphic family of rank $n$ semi-stable complex vector bundles with trivial determinant parametrized by an open neighborhood $V$ of ${\mathcal M}$ in $X$. This gives a holomorphic map from $V$ to ${\mathcal M}_\sigma$, which therefore restricts to a real analytic isomorphism from ${\mathcal M}$ to ${\mathcal M}_\sigma$, showing that the Narasimhan-Seshadri identification is real-analytic. Concerning the pre-quantum Chern-Simons line bundle, we see that it is the restriction of a holomorphic line bundle to ${\mathcal M}$, thus also real analytic. Further, the Takhtajan-Zograf formula \cite{ZT87,ZT88,ZT89,ZT91} for its Hermitian structure proves its real analyticity.

 Let $\varphi$ be a mapping class group element, let $T_{\varphi}$ the mapping torus of $(\Sigma, \varphi),$ and let $L\subset T_{\varphi}$ be the link traced out by the puncture of $\Sigma$. Let $\tilde{\varphi}$ be the lift to the Chern-Simons line bundle, covering the action on $\M.$ It is shown in \cite{AJHM}, that there is a unique choice of $\tilde{\varphi}$ such that that on a component $Y$ of the fixed point set $\M^{\varphi},$ the lift is given by multiplication by $\exp(2\pi i \theta),$ where $\theta$ is a value of the Chern-Simons functional on the moduli space  $\M_{T_{\varphi}}.$ For any $x \in \rho^{-1}(Y),$ where $\rho$ is the natural surjective map $\M_{T_{\varphi}}\rightarrow \M^{\varphi},$ we have $\theta=S_{\text{CS}}(x).$
 
 \begin{proof}[Proof of Theorem \ref{1.1} and of Theorem \ref{MainTheorem2}] Theorem \ref{1.1} follows from Theorem \ref{QIasOI} and Lemma \ref{Prop1}, whereas Theorem \ref{MainTheorem2} follows from Theorem \ref{Theorem2} and Theorem \ref{Main Theorem2}.  \end{proof}

\subsection{The $\text{SU}(2)$ moduli space of a punctured torus}

Let $\Sigma_1^1$ be a torus with a puncture, with mapping class group $\Gamma_{(\bar\Sigma_1^1, v_P)}.$ Let $a,b$ be the standard generators of $\pi_1(\Sigma_1^1).$ For $l\in [-2,2]$ let $\M_l$ be the moduli space of flat $\text{SU}(2)$ bundles over $\Sigma_1^1,$ for which the trace of the holonomy along a small loop encircling the puncture is $l.$ The trace coordinates on the character variety construction of the full moduli space gives an identification
\begin{equation}
\mathcal{M}_l \overset{\sim}{\longrightarrow} \{ (x,y,z) \in \R^3: x^2+y^2+z^2-2-xyz=l\}.
\end{equation}
For more details, see Goldman's article \cite{Goldman97}. For $l=-2,$ the space $\mathcal{M}_l$ consists of a single point, for $l\in (-2,2)$ it is smooth and diffeomorphic to $S^2,$ while  $\M_2$ is isomorphic to $\mathbb{T}^2/\{\pm1 \},$ and contains four singularities. This case is known as the pillow case. It is clear that the action of $\Gamma_{(\bar\Sigma_1^1, v_P)}$  resctricts to an action on $\mathcal{M}_l$ for all $l.$ Brown has shown \cite{Brown08} that if $\varphi$ is a pseodo-Anosov homeomorphism, then the fixed point set $\mathcal{M}^{\varphi}$ is a real algebraic subset of dimension $1,$ that meets $\mathcal{M}_l$ transversely for almost all $l.$

Let $t_a \in \Gamma_{(\bar\Sigma_1^1, v_P)}$  be the Dehn-twist about $a,$ and let $t_b \in \Gamma_{(\bar\Sigma_1^1, v_P)}$  be the Dehn-twist about $b.$ We have a explicit expression of the action of $t_a,t_b$ on the moduli space $\text{M}_l$ in terms of the coordinates $(x,y,z)$ introduced above. As explained in \cite{Goldman97} we have 
\begin{equation} \label{coordinateexpressionsforDehntwists}
t_a(x,y,z)=(x,z,xz-y), \ \ \ t_b(x,y,z)=(xy-z,y,x).
\end{equation}
Consider the composition
\begin{equation} \label{mappeclass}
\varphi=t_b^{-1}\circ t_a.
\end{equation}
It follows by Penner's work \cite{Penner88} that $\varphi$ is a pseudo-Anosov homeomorphism. 

We shall consider the action of $\varphi,\varphi^2$ on $\mathcal{M}_{-1/4}.$ Let
\begin{equation}
p=\left(-1,\frac{1}{2},\frac{1}{2}\right),
\ \ \ q= \left(\frac{1}{2},-1,-1\right).
\end{equation}
The fixed point set of $\varphi$ as well as the fixed point set of $\varphi^2$ is equal to $\{p,q\},$ and we have
\begin{align} \label{fixvarphi1}  \text{dim}(\text{Ker}(\dv \varphi_p-I)) &= \text{dim}(\text{Ker}(\dv \varphi_q-I))=0, 
\\ \text{dim}(\text{Ker}(\dv \varphi^2_p-I))&=\text{dim}(\text{Ker}(\dv \varphi^2_q-I))=1.
\end{align}
These assertions are proven by a direct calculation presented in the Appendix. Thus we have identified a pseudo-Anosov homeomorphism $\varphi$ whose fixed point set is cut out transversely, whereas $\varphi^2$ satisfy the conditions of Theorem \ref{MainTheorem2} but whose fixed point set it not cut out transversely. 

The general construction of a Hitchin connection in \cite{Andersen12} applies to the case of $\mathcal{M}_l$, for all $l\in (-2,2)$, with its Chern-Simons line bundle constructed in \cite{AJHM} and its family of complex structures parametrized by Teichm\"{u}ller space $\mathcal{T}$ of $\Sigma_1^1$ constructed in the works of Daskalopoulos and Wentworth and Mehta and Seshadri \cite{DaskalopoulosWentworth,MehtaSeshadri}. The mapping class group $\Gamma_{(\bar\Sigma_1^1, v_P)}$  acts on this setup and the same argument as in the co-prime case shows that this is in fact a real analytic action. Further, we get in this case, since the complex dimension of $\mathcal{T}$ is one, by Theorem 4.8 of \cite{AndersenGammelgaard14} that any of the Hitchin connections provided by the construction of the first named author in \cite{Andersen12} is projectively flat. Picking one of these projectively flat Hitchin connections, we get a projective representation of $\Gamma_{(\bar\Sigma_1^1, v_P)}.$ For any $\varphi\in \Gamma_{(\bar\Sigma_1^1, v_P)}$  and any $l\in (-2,2),$ let
$$Z_l^{(k)}(\varphi) \in GL(H_\sigma^{(k)}),$$  
be any linear lift for some $\sigma\in \mathcal{T}$, as constructed in Section \ref{Holomorphic quantization} for the situation at hand. In this case, parallel transport also has an asymptotic expansion in terms of Toeplitz operators as in \eqref{VIP}. 

\begin{proof}[Proof of Theorem \ref{NiceResult}] This is an application of Theorem \ref{Theorem2} and Theorem \ref{Main Theorem2} to this situation, together with the proof of Proposition $5.1$ in \cite{Brown08}, which shows that for generic $l,$ $\M_{l}^{\varphi}$ is finite and transversely cut out.\end{proof}

\section*{Appendix} \label{Appendix}

We analyze the fixed point set of the mapping class \eqref{mappeclass} in detail in this appendix. Introduce the function $L(x,y,z)=x^2+y^2+z^2-2-xyz,$ such that $\M_l$ is the level-set $L(x,y,z)=l.$ We start by giving a coordinate expression for the action of $\varphi.$ We observe that $t_b^{-1}(x,y,z)=(z,y,zy-x),$ as is easily deduced from \eqref{coordinateexpressionsforDehntwists}. From this we get the expression
\begin{equation} \label{coordinateexpressionforvarphi}
\varphi(x,y,z)=(xz-y,z,z(xz-y)-x). 
\end{equation}
From the coordinate expression \eqref{coordinateexpressionforvarphi} one verifies that
\begin{equation}
\varphi(p)=p, \ \ \ \varphi(q) =q.
\end{equation}
As any fixed point of $\varphi$ is a fixed point of $\varphi^2,$ we can prove the set-theoretical identities by arguing that
\begin{equation}
\mathcal{M}_{-1/4}^{\varphi^2} \subset \{p,q\}.
\end{equation}
Observe that $q'$ is a fixed point of $\varphi^2$  if and only if
\begin{equation} 
\varphi(q')=\varphi^{-1}(q').
\end{equation}
From the coordinate expression \eqref{coordinateexpressionforvarphi} one sees that
\begin{equation} \label{expofinveser}
\varphi^{-1}(x,y,z)=(yx-z,y(yx-z)-x,y).
\end{equation}
Thus we see by comparing that $(x,y,z)$ is a fixed point of $\varphi^2$ if and only if
\begin{equation} \label{fixedpointsequ.}
xz-y=yx-z, \  y(yx-z)-x =z, \  z(zx-y)-x =y.
\end{equation}
The first of the equations appearing in \eqref{fixedpointsequ.} is equivalent to
\begin{equation} \label{x(z-y)=y-z}
x(z-y)=y-z.
\end{equation}
Now, let $(x,y,z)$ be a fixed point. From \eqref{x(z-y)=y-z} we see that we must either have that $x=-1,$ or that $z=y.$

 We shall start by assuming that $x=-1.$ Hence the two lower equations appearing in \eqref{fixedpointsequ.} reads
\begin{equation} \label{temlighemlig}
-y^2-yz+1=z, \ -z^2-yz+1=y,
\end{equation}
from which we learn that
\begin{equation}
y-z=-z^2-yz+1-(-y^2-yz+1)=y^2-z^2=(y-z)(y+z).
\end{equation}
Hence we can assume that either $y=z$ or $y+z=1.$ If $y=z$ we have
\begin{align}
0=L(-1,y,y)+\frac{1}{4}=3y^2-\frac{3}{4}.
\end{align}
This equation has exactly two solutions given by $y=-1/2$ and $y=1/2.$
However, we see that the identities $ y=z=-1/2$ would violate \eqref{temlighemlig}. It follows that we must have $y=z=\frac{1}{2},$ which corresponds to the solution $p.$ Assume now that $z+y=1.$ Then we have
\begin{align}
0=L(-1,y,1-y)+\frac{1}{4}=y^2-y+1/4.
\end{align}
This equation has only one solution $y=1/2,$ which again corresponds to the point $p.$

We now assume that $z=y.$ The second equation of \eqref{fixedpointsequ.} now reads $z^2x-z^2-x=z.$ This is equivalent to $(z^2-1)x=z+z^2,$ which can be rewritten as
\begin{equation} \label{rerere}
(z+1)(z-1)x=z(z+1).
\end{equation}
Observe that this implies that $z\not=1.$ If $z=-1,$ we get
\begin{align}
0=L(x,-1,-1)+\frac{1}{4}=x^2-x+1/4.
\end{align}
This has precisely one solution namely $x=1/2,$ and we recover the solution $p.$  Assume now that $z\not=-1.$ Then \eqref{rerere} implies that
\begin{equation} \label{niahahahah}
x=\frac{z}{z-1}.
\end{equation}
Thus we get
\begin{align}
0=L(z/(z-1),z,z)+\frac{1}{4}= \frac{(z+1)(2z-1)(2z^2-7z+7)}{(z-1)^2}.
\end{align}
This equation has exactly two real solutions $z=-1$ and $z=1/2.$ As we have discarded $-1,$ this implies that $z=1/2,$ and from  \eqref{niahahahah} we conclude $x=-1.$ This corresponds to the point $q.$

We have now proven the claims about the fixed point sets \eqref{fixvarphi1},  and it remains to compute the differentials. We start with some general considerations. From \eqref{coordinateexpressionforvarphi} we compute that with respect to the $\partial x, \partial y, \partial z$ basis of $T \R^3$ we have
\begin{equation} \label{Differential}
\dv \varphi =\begin{pmatrix} z & -1 & x
\\ 0 & 0 & 1
\\ z^2-1 & -z & 2xz-y \end{pmatrix}.
\end{equation}
We have
\begin{equation} \label{Tangentspace}
T_p \mathcal{M}_{-1/4}=\text{Ker}(\dv L),
\end{equation}
and
\begin{equation} \label{DL}
\dv L=( 2x-yz,2y-xz,2z-xy).
\end{equation}

We now turn to the point $p.$ As
\begin{equation}
\dv L_p=\left( 0,\frac{-3}{2}, \frac{-3}{2}\right),
\end{equation}
we conclude that
\begin{equation}
T_p \mathcal{M}_{-1/4}= \R \partial x \oplus \R(\partial y-\partial z).
\end{equation}
With respect to the basis $\partial x, \partial y, \partial z$ of $T_p \R^3$ we have
\begin{equation} 
\dv \varphi_p  = \begin{pmatrix} -1 & -1 & 1/2
\\ 0 & 0 & 1
\\ 0 & 1 & 0 \end{pmatrix}.
\end{equation}
Introduce the basis 
\begin{equation}
v_1 =\partial x, \ \ \  v_2 =\partial y- \partial z.
\end{equation}
With respect to the $(v_1,v_2)$ basis we have
\begin{equation}
\dv \varphi_p = \begin{pmatrix} -1 & -3/2 & 
\\ 0 & -1 
\end{pmatrix}.
\end{equation}
It follows that $1$ is not an eigenvalue of $\dv \varphi_p.$ We now consider $\dv \varphi^2_p.$ With respect to the basis $(v_1,v_2)$ it has matrix given by
\begin{equation}
\dv \varphi^2_p = \begin{pmatrix} 1 & 3
\\ 0 & 1
\end{pmatrix}.
\end{equation}
It follows that
\begin{equation}
\text{Dim}(\text{Ker}(\dv \varphi^2_p-I))=1.
\end{equation}

We now turn to the point $q. $ We have
\begin{equation}
\dv \varphi_q = \begin{pmatrix} 1/2 & -1 & -1
\\ 0 & 0 & 1
\\ -3/4 & -1/2 & -3/2 \end{pmatrix}.
\end{equation}
We have
\begin{equation}
\dv L_q= \left(\frac{-9}{4},\frac{3}{2}, \frac{3}{2}\right).
\end{equation}
Thus we have 
\begin{equation}
T_q \mathcal{M}_{-1/4}=\R(\partial y-\partial z)\oplus \R(\partial x+3/2\partial y).
\end{equation}
Introduce the basis
\begin{equation}
w_1 =\partial y-\partial z, \ \ \  w_2 = \partial x+3/2\partial y.
\end{equation}
We have
\begin{equation}
\dv \varphi_q (w_1)=-w_1, \ \ \ \dv \varphi_q (w_2)=3w_1-w_2.
\end{equation}
Thus we see that with respect to $w_1,w_2,$ we have \begin{equation}
\dv \varphi_q = \begin{pmatrix} -1 & 3 
\\ 0 & -1 
\end{pmatrix},
\ \ \ \dv \varphi^2_q \sim \begin{pmatrix} 1 & -6 
\\ 0 & 1
\end{pmatrix}.
\end{equation}
This concludes the analysis.

%

\bibliography{bibo2}
\
\

\noindent 
       Jørgen Ellegaard Andersen \\
      Center for Quantum Geometry of Moduli Spaces\\
      Department of Mathematics\\
        University of Aarhus\\
        DK-8000, Denmark\\
     andersen{\@@}qgm.au.dk
     \\\\
     William Elbæk Petersen \\
      Center for Quantum Geometry of Moduli Spaces\\
      Department of Mathematics\\
        University of Aarhus\\
        DK-8000, Denmark\\
    william{\@@}qgm.au.dk

\end{document}